\documentclass[11pt,reqno]{article}
\usepackage[utf8]{inputenc}
\usepackage[english]{babel} 

\usepackage{amsmath}
\usepackage{amssymb}
\usepackage{amsthm}
\usepackage{amscd}

\usepackage{cite}

\usepackage[matrix,arrow,curve]{xy}

\begin{document}
\righthyphenmin=2

\renewcommand{\refname}{References}
\renewcommand{\bibname}{Bibliography}
\renewcommand{\proofname}{Proof}

\newtheorem{lm}{Lemma}
\newtheorem{tm}{Theorem}
\newtheorem*{prop}{Proposition}
\newtheorem*{mtm}{Main Theorem}
\newtheorem*{atm}{Another Presentation}
\newtheorem{cl}{Corollary}
\newtheorem{mcl}{Corollary}
\newtheorem*{cl*}{Corollary}
\theoremstyle{definition}
\newtheorem*{df}{Definition}
\theoremstyle{remark}
\newtheorem*{rk}{Remark}

\newcommand{\Card}{\mathop{\mathrm{Card}}\nolimits}
\newcommand{\Ker}{\mathop{\mathrm{Ker}}\nolimits}
\newcommand{\Cent}{\mathop{\mathrm{Cent}}\nolimits}
\newcommand{\E}{{\mathrm{E}}}
\newcommand{\St}{\mathop{\mathrm{St}}\nolimits}
\newcommand{\Sp}{\mathop{\mathrm{Sp}}\nolimits}
\newcommand{\Ep}{\mathop{\mathrm{Ep}}\nolimits}
\newcommand{\GL}{\mathop{\mathrm{GL}}\nolimits}
\newcommand{\Kt}{\mathop{\mathrm{K_2}}\nolimits}
\newcommand{\Ko}{\mathop{\mathrm{K_1}}\nolimits}
\newcommand{\Ho}{\mathop{\mathrm{H_1}}\nolimits}
\newcommand{\Ht}{\mathop{\mathrm{H_2}}\nolimits}
\newcommand{\epi}{\twoheadrightarrow}
\newcommand{\sgn}{\mathrm{sgn}}
\newcommand{\eps}[1]{\varepsilon_{#1}}
\newcommand{\lan}{\langle}
\newcommand{\ran}{\rangle}
\newcommand{\inv}{^{-1}}
\newcommand{\ur}[1]{\!\,^{(#1)}U_1}
\newcommand{\ps}[1]{\!\,^{(#1)}\!P_1}
\newcommand{\ls}[1]{\!\,^{(#1)}\!L_1}

\renewcommand{\labelenumi}{\theenumi{\rm)}}
\renewcommand{\theenumi}{\alph{enumi}}

\author{Andrei Lavrenov\footnote{The author acknowledges support of the State Financed research task 6.38.191.2014 ``Structure theory, classification, geometry, arithmetic and K-theory of algebraic groups and related structures'' at the St. Petersburg State University, the Chebyshev Laboratory  (Department of Mathematics and Mechanics, St. Petersburg State University)  under RF Government grant 11.G34.31.0026, JSC ``Gazprom Neft'', RFBR project 13-01-00709 ``Study of algebraic groups over rings by localization methods'', RFBR project 13-01-92699 ``Classical algebraic K-theory and algebraic groups'' and the M\"obius Contest Foundation for Young Scientists.}}

\title{Another presentation\\ for symplectic Steinberg groups}

\date{}

\maketitle

\begin{abstract}
We solve a classical problem of centrality of symplectic $\Kt$, namely we show that for an arbitrary commutative ring $R$, $l\geq3$, the symplectic Steinberg group $\St\!\Sp(2l,\,R)$ as an extension of the elementary symplectic group $\Ep(2l,\,R)$ is a central extension. This allows to conclude that the explicit definition of symplectic $\Kt\!\Sp(2l,\,R)$ as a kernel of the above extension, i.e. as a group of non-elementary relations among symplectic transvections, coincides with the usual implicit definition via plus-construction.

We proceed from van der Kallen's classical paper, where he shows an analogous result for linear K-theory. We find a new set of generators for the symplectic Steinberg group and a defining system of relations among them. In this new presentation it is obvious that the symplectic Steinberg group is a central extension.
\end{abstract}

\section*{Introduction}

The main result of the present paper is centrality of $\Kt\!\Sp(2l,\,R)$ in the symplectic Steinberg group over an arbitrary commutative ring $R$ and $l\geq3$.

Centrality of the usual linear $\Kt(n,\,R)$ for $n\geq4$ was proven by van der Kallen in~\cite{vdK} and Tulenbaev in~\cite{Tulelem} for a commutative and an almost commutative ring $R$, respectively. Bak and Tang announced in 1998 that they work on a similar result for unitary groups, but it is a huge project and their proofs have not yet been published. Their results would imply ours for $l\geq5$, however our proof works for $l\geq3$ as well. As shown by Wendt in~\cite{Wendt}, centrality of the symplectic $\Kt$ does not hold for $l=2$.

Centrality of symplectic $\mathrm K_2$ is known at the stable level, i.e., for the limit groups $\Kt\!\Sp(R)=\varinjlim\Kt\!\Sp(2l,\,R)$, or, in fact, once surjective stability of $\mathrm K_2$ starts (see~\cite{MSgrc} and also~\cite{JMb,HOb}). In~\cite{Steinstab} it is shown that surjective stability holds for rings whose stable rank is small enough with respect to the rank $l$ of the root system, for instance, for local rings (see also~\cite{HOb,VP2,SCh,Stepkt}). In~\cite{Stepkt} Stepanov shows in particular that $\mathrm K_2$ is central for a ring $R$ if and only if it is central for the factor-ring $R/\mathrm{Rad}\,R$, where $\mathrm{Rad}\,R$ denotes the Jacobson radical of $R$.

The starting point of our approach is van der Kallen's beautiful idea of ``another presentation'' for Steinberg groups~\cite{vdK}. The large part of our techniques is a symplectic analogue of his results. We find another set of generators for the symplectic Steinberg group and establish a nice behaviour of these new generators under conjugation. After that the claim is immediate. Another important idea used in the proof was suggested by Vavilov and Stepanov in~\cite{StepVav2000}. They notice that the elementary symplectic group can be generated by long-root unipotents. Similarly, the symplectic Steinberg group is generated by long-root unipotents, defined in the present paper. 

The main result of the present paper can be stated as follows.

\begin{mtm}
Let $R$ be an arbitrary commutative ring, $l\geq3$, let $\phi$ denote the natural projection $\phi\colon\St\!\Sp(2l,\,R)\epi\Ep(2l,\,R)$, and let $\Kt\!\Sp(2l,\,R)$ be its kernel $\Ker\phi$. Then one has 
$$\Kt\!\Sp(2l,\,R)\leq\Cent\St\!\Sp(2l,\,R),$$ 
or, in other words, $\phi$ is a central extension.
\end{mtm}

As a consequence one can conclude that for $l\geq4$ the group $\Kt\!\Sp(2l,\,R)$ is the Schur multiplier of the elementary symplectic group, $\Ht(\Ep(2l,\,R),\,\mathbb Z)$. This follows from a standard fact of the theory of central extensions. Namely if $\pi\colon G\epi H$ is a central extension such that $G$ is super-perfect (i.e. $\Ho(G,\,\mathbb Z)=\Ht(G,\,\mathbb Z)=0$), then $\Ht(H, \mathbb Z)=\Ker\pi$. The proof of this fact immediately follows from the Lyndon--Hochschild--Serre spectral sequence (see~\cite{JMb,RSb} for details). It is well-known that various Steinberg groups are super-perfect under appropriate conditions (see~\cite{ABb,AB2,HOb,vdK2,vdKSosm,KMsog,KMusg,L1,L2,JMb,MSgrc,S2,RSb,GThgk}). In particular, Stein shows in~\cite{MSgrc} that $\St\!\Sp(2l,\,R)$ is super-perfect for $l\geq4$. Thus, the Main Theorem implies the following result.

\begin{mcl}
In notation of the Main Theorem assume in addition that $l\geq4$. Then one has 
$$
\Kt\!\Sp(2l,\,R)=\Ht(\Ep(2l,\,R),\,\mathbb Z).
$$
\end{mcl}

The above corollary can be generalised to the case $l=3$. More precisely, consider a ``universal covering'' of $\St\!\Sp(6,\,R)$, i.e. a central extension $\pi\colon U\epi\St\!\Sp(6,\,R)$ such that $U$ is super-perfect (it exists since $\St\!\Sp(6,\,R)$ is perfect, see~\cite{JMb,RSb} for details). Then $\phi\pi\colon U\epi\Ep(6,\,R)$ is itself a central extension, so that we have the following diagram
$$
\xymatrix{
1\ar@{->}[r]&\Ht(\St\!\Sp(6,\,R),\mathbb Z)\ar@{->}[r]\ar@{->}[d]&U\ar@{->}^{\pi\qquad\!}[r]\ar@{=}[d]&\St\!\Sp(6,\,R)\ar@{->}[r]\ar@{->}^\phi[d]&1\\
1\ar@{->}[r]&\Ht(\Ep(6,\,R),\mathbb Z)\ar@{->}[r]&U\ar@{->}[r]&\Ep(6,\,R)\ar@{->}[r]&1
}
$$
with exact rows. Now, the snake lemma provides the exact sequence
$$
\xymatrix{
1\ar@{->}[r]&{\Ht(\St\!\Sp(6,\,R),\mathbb Z)}\ar@{->}[r]&{\Ht(\Ep(6,\,R),\mathbb Z)}\ar@{->}[r]&{\Kt\!\Sp(6,\,R)}\ar@{->}[r]&1.
}
$$
Since $\Ht(\St\!\Sp(6,\,R),\mathbb Z)$ is computed in~\cite{S2,vdKSosm} for an arbitrary $R$, one can now compare $\Kt\!\Sp(2l,\,R)$ and $\Ht(\Ep(2l,\,R),\,\mathbb Z)$ for $l=3$. Using results of~\cite{vdKSosm} one obtains the following generalisation of the previous corollary.

\begin{mcl}
In notation of the Main Theorem there is a surjective map from $\Ht(\Ep(2l,\,R),\,\mathbb Z)$ onto $\Kt\!\Sp(2l,\,R)$ with the kernel $\Ht(\St\!\Sp(2l,\,R),\mathbb Z)$. In particular, this map is bijective for $l\geq4$ or for $l=3$ provided that $R$ has no residue field isomorphic to $\mathbb F_2$.
\end{mcl}

The above corollaries establish coincidence of the classical definition of $\Kt\!\Sp$ as above with Quillen's symplectic K-theory. Recall that given a connected based CW-complex $X$ and a perfect subgroup $E$ in $\pi_1(X)=G$ one can construct a new CW-complex $X^+$ by attaching 2-cells and 3-cells such that $\pi_1(X)\rightarrow\pi_1(X^+)$ is the natural projection $G\epi G/E$ and induced maps on homologies are isomorphisms. Quillen used this construction to define higher K-groups. The same construction can be applied to define K-theory of Chevalley groups as well, see~\cite{SCh}. In particular, taking $G=\Sp(2l,\,R)$, $E=\Ep(2l,\,R)$ and $X=\mathrm BG$ one obtains the definition of Quillen's symplectic K-theory
$$
\mathrm K_i\!\Sp^Q(2l,\,R)=\pi_i(\mathrm BG^+).
$$
Such a definition automatically implies that Quillen's $\mathrm K_1\!\Sp^Q(2l,\,R)$ coincides with the classical one $\mathrm K_1\!\Sp(2l,\,R)=\Sp(2l,\,R)/\Ep(2l,\,R)$ (which, as a matter of fact, coincides with Volodin's one). One can show that $\mathrm BE^+$ is homotopy equivalent to the universal covering of $\mathrm BG^+$ (see~\cite{SCh}). Thus, $\mathrm K_2\!\Sp^Q(2l,\,R)=\pi_2(\mathrm BG^+)=\pi_2(\mathrm BE^+)$, so that it is equal to $\Ht(\mathrm BE^+)$ by Hurewicz theorem. Since the plus-construction preserves homologies, and homologies of $\mathrm BE$ are equal to homologies of $E$, one can conclude that $\mathrm K_2\!\Sp^Q(2l,\,R)=\Ht(E,\mathbb Z)$. In other words, the Main Theorem of the present paper allows to compare definitions of $\mathrm K_2$ via Steinberg groups and via the plus-construction.

\begin{mcl}
In notation of the Main Theorem there is a surjective map from $\mathrm K_2\!\Sp^Q(2l,\,R)$ onto $\Kt\!\Sp(2l,\,R)$ with the kernel $\Ht(\St\!\Sp(2l,\,R),\mathbb Z)$. In particular, this map is bijective for $l\geq4$ or for $l=3$ provided that $R$ has no residue field isomorphic to $\mathbb F_2$.
\end{mcl}

The second major result of the present paper is another presentation for the symplectic Steinberg group. As above, it is a symplectic analogue of van~der~Kallen's result for the linear case. More precisely, we have the following theorem.

\begin{atm}
In notation of the Main Theorem the symplectic Steinberg group $\St\!\Sp(2l,\,R)$ can be defined by the set of generators
\begin{multline*}
\big\{X(u,\,v,\,a)\,\big| u,\,v\in V,\ \text{$u$ is a column of}\\ \text{a symplectic elementary matrix},\ \lan u,\,v\ran=0,\ a\in R\big\}
\end{multline*}
and relations
\setcounter{equation}{0}
\renewcommand{\theequation}{P\arabic{equation}}
\begin{align}
&X(u,\,v,\,a)X(u,\,w,\,b)=X(u,\,v+w,\,a+b+\lan v,\,w\ran),\\
&\begin{aligned}X(u,\,va,\,0)=X(v,\,ua,\,0)\,\text{ where}&\text{ $v$ is also a column}\\ &\text{of a symplectic elementary matrix,}\end{aligned}\\
&\begin{aligned}X(u',\,v',\,b)X(u,\,v,\,a)X(u',\,v',\,b&)\inv=\\&=X(T(u',\,v',\,b)u,\,T(u',\,v',\,b)v,\,a),\end{aligned}
\end{align}
where $T(u,\,v,\,a)$ is an ESD-transformation $$w\mapsto w+u(\lan v,\,w\ran+a\lan u,\,w\ran)+v\lan u,\,w\ran.$$ For the usual generators of the symplectic Steinberg group the following identities hold 
$$
X_{ij}(a)=X(e_i,\,e_{-j}a\eps{-j},\,0)\ \text{for $j\neq-i$},\qquad X_{i,-i}(a)=X(e_i,\,0,\,a),
$$
and, moreover, $\phi$ sends $X(u,\,v,\,a)$ to $T(u,\,v,\,a)$. Furthermore, the following relations are automatically satisfied
\begin{align}
&X(u,\,ua,\,0)=X(u,\,0,\,2a),\\
&X(ub,\,0,\,a)=X(u,\,0,\,ab^2),\\
&X(u+v,\,0,\,a)=X(u,\,0,\,a)X(v,\,0,\,a)X(v,\,ua,\,0)\,\text{ for }\lan u,\,v\ran=0,
\end{align}
whenever $ub$ in {\rm P5} or $v$ and $u+v$ in {\rm P6} are also columns of symplectic elementary matrices.
\end{atm}

Formally, the first theorem follows from the second one, but the proof of the above presentation crucially depends on all intermediate results, necessary for the Main Theorem.

I would like to express my gratitude to Nikolai Vavilov, who supervised this work and also suggested numerous improvements in the exposition. I would also like to thank Alexei Stepanov for helpful discussions, Sergey Sinchuk, who partially motivated me to start this research, and Alexander Luzgarev, Andrei Smolensky and Alexander Shchegolev for their careful reading of the drafts of this paper.

\section{Notations}

In the sequel $R$ denotes an arbitrary associative commutative unital ring, $V=R^{2l}$ denotes a free right $R$-module with basis numbered $e_{-l}$, $\ldots$, $e_{-1}$, $e_1$, $\ldots$, $e_l$, $l\geq3$. For the vector $v\in V$ its $i$-th coordinate will be denoted by $v_i$, i.e. $v=\sum_{i=-l}^le_iv_i$. By $\lan\ ,\ \ran$ we denote the standard symplectic form on $V$, i.e $\lan e_i,\,e_j\ran=\sgn(i)\delta_{i,-j}$. We will usually write $\eps i$ instead of $\sgn(i)$. Observe that $\lan u,\,u\ran=0$ for any $u\in V$.

\begin{df}
Define the {\it symplectic group} $\Sp(V)=\Sp(2l,\,R)$ as the group of automorphisms of $V$ preserving the symplectic form $\langle\ ,\ \rangle$,
$$\Sp(V)=\{f\in\GL(V)\mid\lan f(u),\,f(v)\ran=\lan u,\,v\ran\ \ \forall\,u,\,v\in V\}.$$
\end{df}

\begin{df}[Eichler--Siegel--Dickson transformations]
For $a\in R$ and $u$, $v\in V$, $\lan u,\,v\ran=0$, denote by $T(u,\,v,\,a)$ the automorphism of $V$ s.t. for $w\in V$ one has
$$T(u,\,v,\,a)\colon w\mapsto w+u(\lan v,\,w\ran+a\lan u,\,w\ran)+v\lan u,\,w\ran.$$
We refer to the elements $T(u,\,v,\,a)$ as the ({\it symplectic}) {\it ESD-trans\-for\-ma\-tions}.
\end{df}

This definition follows Petrov's paper~\cite{Petoug}. It simultaneously generalises short-root unipotents $T_{u,v}(a)=T(u,va,0)$ and long-root unipotents $T_u(a)=T(u,0,a)$. In general, an ESD-transformation is a product of a long-root unipotent and a short-root one, $T(u,\,v,\,a)=T_u(a)\,T_{u,v}(1)$.

The following properties of ESD-transformations can be verified by a straightforward computation.

\begin{lm}
\label{esd-properties}
Let $u$, $v$, $w\in V$ be three vectors such that $\lan u,\,v\ran=0$, $\lan u,\,w\ran=0$, and let $a$, $b\in R$. Then
\begin{enumerate}
\item $T(u,\,v,\,a)\in\Sp(V)$,
\item $T(u,\,v,\,a)\,T(u,\,w,\,b)=T(u,\,v+w,\,a+b+\lan v,\,w\ran)$,
\item $T(u,\,va,\,0)=T(v,\,ua,\,0)$,
\item $g\,T(u,\,v,\,a)g\inv=T(gu,\,gv,\,a)\ \ \forall\,g\in\Sp(V)$.
\end{enumerate}
\end{lm}

\begin{rk}
Observe that $T(u,\,0,\,0)=1$ and $T(u,\,v,\,a)\inv=T(u,\,-v,\,-a)$.
\end{rk}

In the following particular case we have a simple commutator formula for ESD-transformations. In the present paper all commutators are left-normed, $[x,\,y]=xyx\inv y\inv$, we denote $xyx\inv$ by $\!\,^xy$.

\begin{lm}
\label{z-decomposition}
For $u$, $v\in V$ such that $u_i=u_{-i}=v_i=v_{-i}=0$, $\lan u,\,v\ran=0$, and $a\in R$ one has 
$$
[T(e_i,\,u,\,0),\,T(e_{-i},\,v,\,a)]=T(u,\,v\eps{i},\,a)T(e_{-i},\,-ua\eps{-i},\,0).
$$
\end{lm}

\begin{proof}
The proof is a direct computation using properties of ESD-trans\-vec\-tions stated above. Firstly, using the conjugation formula one can rewrite the commutator in the following way:
\begin{align*}
&[T(e_i,\,u,\,0),\,T(e_{-i},\,v,\,a)]=T(e_i,\,u,\,0)\cdot\,^{T(e_{-i},\,v,\,a)}T(e_i,\,-u,\,0)=\\
&=T(e_i,\,u,\,0)T(T(e_{-i},\,v,\,a)e_i,\,-T(e_{-i},\,v,\,a)u,\,0)=\\
&=T(e_i,\,u,\,0)T(e_i+e_{-i}a\eps{-i}+v\eps{-i},\,-u,\,0).
\end{align*}
Now, using {\it b}) and {\it c}) of Lemma~\ref{esd-properties} one gets
\begin{align*}
&T(e_i,\,u,\,0)T(e_i+e_{-i}a\eps{-i}+v\eps{-i},\,-u,\,0)=\\
&=T(u,\,e_i,\,0)T(u,\,-e_i+e_{-i}a\eps i+v\eps{i},\,0)=T(u,\,e_{-i}a\eps i+v\eps{i},\,a)=\\&=T(u,\,v\eps{i},\,a)T(u,\,e_{-i}a\eps i,\,0)=T(u,\,v\eps{i},\,a)T(e_{-i},\,-ua\eps{-i},\,0).
\end{align*}
\end{proof}

Next, we define the elementary symplectic group and its combinatorial analogue, the symplectic Steinberg group, explicitly defined by generators and relations.

\begin{df}
Define $T_{ij}(a)=T(e_{i},\,e_{-j}a\eps{-j},\,0)$ and $T_{i,-i}(a)=T(e_i,\,0,\,a)$, where $a\in R$, $i$, $j\in\{-l,$ $\ldots,$ $-1,$ $1,$ $\ldots,$ $l\}$, $i\not\in\{\pm j\}$. We refer to these elements as the {\it elementary symplectic transvections}. The subgroup of $\Sp(V)$ they generate is called the {\it elementary symplectic group}
$$\Ep(V)=\Ep(2l,\,R)=\lan T_{ij}(a)\mid i\neq j,\ a\in R\ran\leq\Sp(V).$$
\end{df}

Our choice of long-root elementary transvections $T_{i,-i}(a)$ coincides with that in~\cite{Petoug}. The usual choice of sign differs from ours, namely for $2\chi_i\in C_l$ a long root, the corresponding elementary transvection is $x_{2\chi_i}(a)=T_{i,-i}(a\eps i)$. Notice also that $T_{ii}(a)$ are not defined. 

\begin{df}
The {\it symplectic Steinberg group} $\St\!\Sp(2l,\,R)$ is the group generated by the formal symbols $X_{ij}(a)$, $i\neq j$, $a\in R$ subject to the Steinberg relations
\setcounter{equation}{-1}
\renewcommand{\theequation}{R\arabic{equation}}
\begin{align}
&X_{ij}(a)=X_{-j,-i}(-a\eps i\eps j),\\
&X_{ij}(a)X_{ij}(b)=X_{ij}(a+b),\\
&[X_{ij}(a),\,X_{hk}(b)]=1,\text{ for }h\not\in\{j,-i\},\ k\not\in\{i,-j\},\\
&[X_{ij}(a),\,X_{jk}(b)]=X_{ik}(ab),\text{ for }i\not\in\{-j,-k\},\ j\neq-k,\\
&[X_{i,-i}(a),\,X_{-i,j}(b)]=X_{ij}(ab\eps i)X_{-j,j}(-ab^2),\\
&[X_{ij}(a),\,X_{j,-i}(b)]=X_{i,-i}(2ab\eps i).
\end{align}
Clearly, we assume that both sides in each relation are defined, in particular, that undefined elements of type $X_{ii}(a)$ do not appear. For example, R3 and R4 can be written only for $i\neq k$, thus we do not have any Steinberg relation compraising $[X_{kj}(a),\,X_{jk}(b)]$.
\end{df}

The next lemma is a straightforward consequence of Lemmas~\ref{esd-properties} and \ref{z-decomposition}.

\begin{lm}
There is a natural epimorphism $\phi\colon\St\!\Sp(2l,\,R)\epi\Ep(2l,\,R)$ sending the generators $X_{ij}(a)$ to the corresponding elementary transvections $T_{ij}(a)$. In other words, the Steinberg relations hold for the elementary transvections.
\end{lm}

It is well-known that the elementary subgroup is normal in the symplectic group (see~\cite{Kop,Tadsymp}). For the general linear group over an arbitrary commutative ring a similar result was first obtained by Suslin in~\cite{Sus}, see~\cite{Tulelem,Vas,BorVav} for generalisations. Later analogous results were obtained in larger generality: for classical groups in~\cite{KopSus,Vavstr,Vavthes,Stepthes,LiFuanSp,LiFuanO,Vasort,Vassymp,VP2}, for Chevalley groups in~\cite{Tadthes,Tadchev,Vaschev,Suz}, for unitary groups in~\cite{ABb, Bakelem, AB2, VasYouHong,VasMikh,Hazthes,Hazdim,HV,ABNV,Petover,BakTang}. For further generalisations see~\cite{BakVavfunc,Petoug,Stavthes,PetStav,Stepuni}. An overview on this subject can be found in~\cite{HOb,HV,StepVav2000}. In fact, the proof of centrality of $\mathrm K_2$ is based on the same ideas as the proof of normality of the elementary subgroup.

\begin{df}
Consider $\phi$ as a map from the Steinberg group to the symplectic group. Then its cokernel is denoted by $\Ko\!\Sp(2l,\,R)$ and its kernel is denoted by $\Kt\!\Sp(2l,\,R)$,
$$
\xymatrix{
\Kt\!\Sp(2l,\,R)\ \ar@{>->}[r]&\St\!\Sp(2l,\,R)\ar@{->}^{\ \ \phi}[r]&\Sp(2l,\,R)\ar@{->>}[r]&\Ko\!\Sp(2l,\,R).
}
$$
\end{df}

Now, all the necessary notations are introduced and we can outline the proof of the Main Theorem.

The goal is to define specific elements $X(u,\,v,\,a)\in \phi^{-1}T(u,\,v,\,a)$ subject to the defining relations. Firstly, in \S2 we define elements $X(e_i,\,v,\,a)$, i.e., consider a special case, where $u=e_i$ is a base vector. In fact, any element from a unipotent radical of type $U_1$ in $\Ep(2l,\,R)$ has the form $T(e_i,\,v,\,a)$. Further, the corresponding unipotent radicals of the Steinberg group and of the elementary group are isomorphic (Lemma~\ref{unirad}). This specifies the choice of $X(e_i,\,v,\,a)$. After that we verify the desired properties of $X(e_i,\,v,\,a)$ such as those stated in Lemmas~\ref{esd-properties} and \ref{z-decomposition}, see Lemmas~\ref{y-add}, \ref{y-conjugated-by-ps}, \ref{switch}, \ref{z-decomposition-for-y}.

In \S3 we generalise the definition of $X(u,\,v,\,a)$ to a wider class of $u$. Namely, we require that $u_i=u_{-i}=0$. First, observe that the property stated in Lemma~\ref{z-decomposition}, 
$$
[X(e_i,\,u,\,0),\,X(e_{-i},\,v,\,a)]=X(u,\,v\eps{i},\,a)X(e_{-i},\,-ua\eps{-i},\,0),
$$
follows from the desired relations on our ESD-generators P1--P3 (see Lemma~\ref{r3-r5}). In this formula, only $X(u,\,v\eps{i},\,a)$ is not yet defined. Thus, we can specify the choice of $X(u,\,v\eps i,\,a)\in\phi\inv T(u,\,v\eps i,\,a)$ by this formula. It can be done only in situation when $u_i=v_i=u_{-i}=v_{-i}=0$ (the condition of Lemma~\ref{z-decomposition}). Obviously, we have to check that this definition is correct (Lemma~\ref{z-correctness}). We also generalise this definition to the case of an arbitrary $v$ (where correctness is provided by Lemma~\ref{w-correctness}).

In \S4 we pass to an arbitrary $u$. Since $X(u,\,0,\,a)$ generate the Steinberg group (Lemma~\ref{generating}), we can assume that $v=0$. We use the analogue of the relation
$$
T(v+w,\,0,\,a)=T(v,\,0,\,a)T(w,\,0,\,a)T(w,\,va,\,0)
$$
to define ESD-generators of the Steinberg group. Take $w=e_iu_i+e_{-i}u_{-i}$ and $v=u-w$. Then the left hand side of the above equation is equal to $T(u,\,0,\,a)$, and all ESD-transformations at the right hand side are already lifted to the Steinberg group in \S3. Section \S4 is devoted to the proof of correctness (Lemma~\ref{correctness}) and some auxiliary results.

After defining Steinberg long-root unipotents $X(u,\,0,\,a)$, in \S5 we prove the conjugation property
$$
g\,X(u,\,0,\,a)g\inv=X(\phi(g)u,\,0,\,a).
$$
It suffices to consider only the cases $g=X_{i,-i}(b)$ (Lemma~\ref{conjugation-by-long}) and $g=X_{jk}(b)$ (Lemma~\ref{conjugation-by-short}). At this point the Main Theorem immediately follows.

The objective of \S6 is to prove relations P1--P6. We start with P5 (Lemma~\ref{x-long-scalar}). Then we define short-root unipotents $X(v,\,u,\,0)$ by the relation
$$
X(u+v,\,0,\,1)=X(u,\,0,\,1)X(v,\,0,\,1)X(v,\,u,\,0).
$$
After that, we obtain P4 (Lemma~\ref{short-obvious}) and P6 (Lemma~\ref{x-long-is-three-short}). Next, we define $X(u,\,v,\,a)=X(u,\,v,\,0)X(u,\,0,\,a)$ and check P1--P3 (Lemma~\ref{p-relations}).

Finally, in \S7 we define a symplectic van der Kallen group $\St\!\Sp^*(2l,\,R)$ by relations P1--P3 and show that it is actually isomorphic to the usual symplectic Steinberg group.

\section{Unipotent radicals in Steinberg groups}

Our first goal is to define analogues of ESD-transvections $T(u,\,v,\,a)$ in the Steinberg group for the special case when $u$ is a base vector.

\begin{df}
Define the ({\it Steinberg}) {\it unipotent radical} $$\ur i=\lan X_{ij}(a)\mid j\neq i,\ a\in R\ran$$ and the ({\it Steinberg}) {\it parabolic subgroup} $$\ps i=\lan X_{kh}(a)\mid \{h,-k\}\not\ni i,\ a\in R\ran.$$
\end{df}

The next result is well-known. It easily follows from the Steinberg relations.

\begin{lm}[Levi decomposition]
For $g\in\ps i$, $u\in\ur i$ one has 
$$gug\inv\in\ur i.$$
\end{lm}

The next lemma is another obvious consequence of the Steinberg relations.

\begin{lm}
\label{heis}
One has
$$[\ur i,\,\ur i]\leq\lan X_{i,-i}(a)\ran,\qquad [\ur i,\,\lan X_{i,-i}(a)\ran]=1.$$
\end{lm}

In other words, this lemma asserts that the root subgroup $X_{i,-i}$ lies in the centre of $\ur i$, and that the nilpotent class of $\ur i$ is $\leq2$. Now it is easy to see the following corollary.

\begin{cl*}
Every element of $\ur i$ can be expressed in the form $$X_{i,-l}(a_{-l})\ldots X_{i,-1}(a_{-1})X_{i,1}(a_1)\ldots X_{i,l}(a_l).$$
Clearly, we mean that the nonexistent factor $X_{ii}(a_i)$ is omitted in this product.
\end{cl*}

\begin{lm}
\label{unirad}
The restriction of the natural projection $\phi\colon\St\!\Sp(2l,\,R)\epi\Ep(2l,\,R)$ to $\ur i$ is injective
$$\ur i\cong \phi(\ur i).$$
\end{lm}

\begin{proof}
Take an element $x\in\ur i$. Using the above corollary, decompose $x$ as $$x=X_{i,-l}(a_{-l})\ldots X_{i,-1}(a_{-1})X_{i,1}(a_1)\ldots X_{i,l}(a_l).$$ Then $\phi(x)=1$ implies that $a_i=0$ for all $i$.
\end{proof}

\begin{lm}
\label{unipotent-decomposition}
For $v\in V$ such that $v_{-i}=0$ denote $$v_-=\sum_{i<0}e_iv_i\qquad \text{and}\qquad v_+=\sum_{i>0}e_iv_i.$$ Then
\begin{multline*}
T(e_i,\,v,\,a)=T_{i,-i}(a+2v_i-v_i\eps i-\lan v_-,\,v_+\ran)\cdot\\ \cdot T_{-l,-i}(v_{-l}\eps i)\ldots T_{-1,-i}(v_{-1}\eps i)T_{1,-i}(v_1\eps i)\ldots T_{l,-i}(v_l\eps i).
\end{multline*}
\end{lm}

\begin{proof}
Assume that $i>0$, for $i<0$ the proof looks exactly the same. Since $T_{j,-i}(v_j\eps i)=T(e_i,\,e_jv_j,\,0)$ for $j\neq-i$ and $$T_{i,-i}(2v_i)=T(e_i,\,0,\,2v_i)=T(e_i,\,e_iv_i,\,0),$$ one has 
\begin{multline*}
T_{-l,-i}(v_{-l}\eps i)\ldots T_{-1,-i}(v_{-1}\eps i)=\\=T(e_i,\,e_{-l}v_{-l},\,0)\ldots T(e_i,\,e_{-1}v_{-1},\,0)=T(e_i,\,v_-,\,0),
\end{multline*}
and
\begin{multline*}
T_{i,-i}(2v_i)T_{i,-i}(-v_i\eps i)T_{1,-i}(v_{1}\eps i)\ldots T_{l,-i}(v_{l}\eps i)=\\=T(e_i,\,e_{1}v_{1},\,0)\ldots T(e_i,\,e_{l}v_{l},\,0)=T(e_i,\,v_+,\,0).
\end{multline*}
Here factors $T_{i,-i}(-v_i\eps i)$ and $T_{i,-i}(v_i\eps i)$ just cancel each other. So that the right hand side of the desired equality is in fact equal to
\begin{multline*}
T_{i,-i}(a-\lan v_-,\,v_+\ran)T(e_i,\,v_-,\,0)T(e_i,\,v_+,\,0)=\\=T(e_i,\,0,\,a-\lan v_-,\,v_+\ran)T(e_i,v,\lan v_-,\,v_+\ran)=T(e_i,\,v,\,a).
\end{multline*}
\end{proof}

\begin{df}
For $v\in V$ with $v_{-i}=0$, $a\in R$, define $$Y(e_i,\,v,\,a)=(\phi\vert_{\ur i})\inv\big(T(e_i,\,v,\,a)\big).$$
\end{df}

\begin{rk}
By Lemma~\ref{unipotent-decomposition}, $T(e_i,\,v,\,a)$ indeed lies in $\phi(\ur i)$. Moreover, the same lemma provides the following decomposition.
\end{rk}

\begin{lm}
\label{y-decomposition}
For $v\in V$ such that $v_{-i}=0$, $a\in R$, one has
\begin{multline*}
Y(e_i,\,v,\,a)=X_{i,-i}(a+2v_i-v_i\eps i-\lan v_-,\,v_+\ran)\cdot\\ \cdot X_{-l,-i}(v_{-l}\eps i)\ldots X_{-1,-i}(v_{-1}\eps i)X_{1,-i}(v_1\eps i)\ldots X_{l,-i}(v_l\eps i).
\end{multline*}
\end{lm}

\setcounter{cl}{0}

\begin{cl}
In particular, $Y(e_{-j},\,-e_ia\eps j,\,0)=X_{ij}(a)$ for $i\not\in\{\pm j\}$ and $Y(e_i,\,0,\,a)=X_{i,-i}(a)$.
\end{cl}

\begin{cl}
For $j\neq-i$, $v\in V$ such that $v_{-i}=v_{-j}=0$, $a\in R$, one has $Y(e_i,\,v,\,a)\in\ps j$.
\end{cl}

\begin{lm}
\label{y-add}
For $v$, $w\in V$ such that $v_{-i}=w_{-i}=0$ and $a$, $b\in R$, one has 
$$
Y(e_i,\,v,\,a)Y(e_i,\,w,\,b)=Y(e_i,\,v+w,\,a+b+\lan v,\,w\ran).
$$
\end{lm}

\begin{proof}
Obviously, $(v+w)_{-i}=0$, so that the right hand side of this equality is well-defined. Now, it remains to observe that the images of the elements on both sides under $\phi$ coincide.
\end{proof}

\begin{cl*}
One has $Y(e_i,\,0,\,0)=1$ and $\,Y(e_i,\,v,\,a)\inv=Y(e_i,\,-v,\,-a)$.
\end{cl*}

\begin{lm}
\label{y-conjugated-by-ps}
For $g\in\ps i$, $v\in V$ such that $v_{-i}=0$, $a\in R$, one has $$g\,Y(e_i,\,v,\,a)g\inv=Y(e_i,\,\phi(g)v,\,a).$$
\end{lm}

\begin{proof}
First, observe that since $T_{kh}(a)e_i=e_i$ for $i\not\in\{h,-k\}$ one has $\phi(g)e_i=e_i$. Thus, $$\lan e_i,\,\phi(g)v\ran=\lan \phi(g)e_i,\,\phi(g)v\ran=\lan e_i,\,v\ran=0.$$ It follows that $(\phi(g)v)_{-i}=0$ and the right hand side of the desired equation is well-defined. Finally, observe that the images of both sides under $\phi$ coincide.
\end{proof}

\begin{lm}
\label{switch}
For $j\neq-i$, $a\in R$, one has $Y(e_i,\,e_ja,\,0)=Y(e_j,\,e_ia,\,0)$.
\end{lm}

\begin{proof}
For $i=j$ the claim is obvious. Let $i\neq j$, then
$$Y(e_i,\,e_ja,\,0)=X_{-j,i}(a\eps i)=X_{-i,j}(-a\eps j)=Y(e_j,\,e_ia,\,0),$$
where the second equality is by R0.
\end{proof}

The next lemma is an analogue of Lemma~\ref{z-decomposition} for $Y(e_i,\,v,\,a)$'s.

\begin{lm}
\label{z-decomposition-for-y}
For $v\in V$ such that $v_{-j}=v_{k}=v_{-k}=0$, $k\not\in\{\pm j\}$, $a,$ $b\in R$, one has
$$
[Y(e_k,\,e_jb,\,0),\,Y(e_{-k},\,v,\,a)]=Y(e_j,\,vb\eps{k},\,ab^2)Y(e_{-k},\,-e_jab\eps{-k},\,0).
$$
\end{lm}

\begin{proof}
Since $Y(e_k,\,e_jb,\,0)$ lies in $\ur j$ and $Y(e_{-k},\,v,\,a)$ lies in $\ps j$, both sides of the claimed equality lie in $\ur j$. Now it remains to use that their images in $\Ep(2l,\,R)$ coincide by Lemma~\ref{z-decomposition}.
\end{proof}

\begin{cl*}
For $v\in V$ such that $v_{-j}=v_{k}=v_{-k}=0$, $k\not\in\{\pm j\}$, $a,$ $b\in R$, one has the following decomposition
$$
Y(e_j,\,vb,\,ab^2)=[Y(e_k,\,e_jb,\,0),\,Y(e_{-k},\,v\eps{k},\,a)]Y(e_{-k},\,e_jab\eps{-k},\,0).
$$
\end{cl*}

\begin{df}
For $i\not\in\{\pm j\}$, $\alpha\in R^\times$ define 
$$W_{ij}(\alpha)=X_{ij}(\alpha)X_{ji}(-\alpha\inv)X_{ij}(\alpha).$$ 
\end{df}

The following fact is well-known. One can prove it either using the Steinberg relations, or applying conjugation formulae obtained above.

\begin{lm}
For $i\not\in\{\pm j\}$, $k\not\in\{\pm i,\pm j\}$, $\alpha\in R^\times$, $a\in R$, one has
\begin{enumerate}
\item $\,^{W_{ij}(\alpha)}X_{kj}(a)=X_{ki}(\alpha\inv a),$
\item $\,^{W_{ij}(\alpha)}X_{k,-j}(a)=X_{k,-i}(\alpha a\eps i\eps j),$
\item $\,^{W_{ij}(\alpha)}X_{j,-j}(a)=X_{i,-i}(\alpha^2a),$
\item $\,^{W_{ij}(\alpha)}X_{-j,j}(a)=X_{-i,i}(\alpha^{-2}a).$
\end{enumerate}
\end{lm}

\begin{lm}
\label{weyl-conjugation}
Consider $v\in V$ such that $v_i=v_{-i}=v_j=v_{-j}=0$, $i\not\in\{\pm j\}$, $\alpha\in R^\times$. Then
\begin{enumerate}
\item $\,^{W_{ij}(\alpha)}Y(e_j,\,v,\,a)=Y(e_i,\,v\alpha,\,\alpha^2a),$
\item $\,^{W_{ij}(\alpha)}Y(e_{-j},\,v,\,a)=Y(e_{-i},\,v\alpha\inv\eps i\eps j,\,\alpha^{-2}a).$
\end{enumerate}
\end{lm}

\begin{proof}
By Lemma~\ref{y-decomposition} one has
\begin{multline*}
Y(e_j,\,v,\,a)=X_{j,-j}(a-\lan v_-,\,v_+\ran)\cdot\\ \cdot X_{-l,-j}(v_{-l}\eps j)\ldots X_{-1,-j}(v_{-1}\eps j)X_{1,-j}(v_1\eps j)\ldots X_{l,-j}(v_l\eps j).
\end{multline*}
Observe that the elements of the form $X_{-i,-j}(v_{-i}\eps j)$ and $X_{i,-j}(v_{i}\eps j)$ may be omitted in this product since $v_i=v_{-i}=0$. Thus, by the previous lemma we have
\begin{align*}
&\!\,^{W_{ij}(\alpha)}Y(e_j,\,v,\,a)=\\
&=\!\,^{W_{ij}(\alpha)}X_{j,-j}(a-\lan v_-,\,v_+\ran)\cdot\\
&\cdot\!\,^{W_{ij}(\alpha)}\big(X_{-l,-j}(v_{-l}\eps j)\ldots X_{-1,-j}(v_{-1}\eps j)X_{1,-j}(v_1\eps j)\ldots X_{l,-j}(v_l\eps j)\big)=\\
&=X_{i,-i}((a-\lan v_-,\,v_+\ran)\alpha^2)\cdot\\
&\cdot X_{-l,-i}(v_{-l}\alpha\eps i)\ldots X_{-1,-i}(v_{-1}\alpha\eps i)X_{1,-i}(v_1\alpha\eps i)\ldots X_{l,-i}(v_l\alpha\eps i)=\\
&=Y(e_i,\,v\alpha,\,\alpha^2a).
\end{align*}
Similarly, one can prove {\it b}).
\end{proof}

\section{Definition of ESD-generators for vectors having zeros}

In the preceding section we have in particular proven the decomposition of the type stated in Lemma~\ref{z-decomposition} for the Steinberg ESD-generators $X(u,\,v,\,a)$, in the special case where $u=e_i$. Now, we intend to use this decomposition to define the ESD-generators in larger generality, where $u$ is not necessarily equal to a base vector, but has at least one pair of zero coordinates. As always, the main technical issue is to verify the correctness of such a definition.

In the sequel we assume that if the converse is not specified explicitly, then indexes denoted by different letters are neither equal, nor have a zero sum, e.g., $i\not\in\{\pm j\}$.

\begin{df}
For $u$, $v\in V$ such that $\lan u,\,v\ran=0$, $u_i=u_{-i}=v_i=v_{-i}=0$, $a\in R$ denote
$$
Y_{(i)}(u,\,v,\,a)=[Y(e_i,\,u,\,0),\,Y(e_{-i},\,v\eps{i},\,a)]Y(e_{-i},\,ua\eps{-i},\,0).
$$
\end{df}

\begin{rk}
Due to Lemma~\ref{z-decomposition} one has $\phi\big(Y_{(i)}(u,\,v,\,a)\big)=T(u,\,v,\,a)$.
\end{rk}

\begin{rk}
For $v\in V$ with $v_{-j}=v_i=v_{-i}=0$, $a\in R$ one has
$$
Y_{(i)}(e_j,\,v,\,a)=Y(e_j,\,v,\,a)
$$
by Lemma~\ref{z-decomposition-for-y}. Similarly, one can obtain the following result.
\end{rk}

\begin{lm}
For $v\in V$ with $v_{-j}=v_i=v_{-i}=0$, $b\in R$ one has
$$
Y_{(i)}(v,\,e_jb,\,0)=Y(e_j,\,vb,\,0).
$$
\end{lm}

\begin{proof}
Since $Y(e_{i},\,v,\,0)\in\ps j$, both sides lie in $\ur j$.
\end{proof}

\begin{lm}
\label{z-correctness}
Consider $u$, $v\in V$ such that $\lan u,\,v\ran=0$, $u_i=u_{-i}=u_{j}=u_{-j}=0$ and $v_i=v_{-i}=v_j=v_{-j}=0$, and $a\in R$. Then one has
$$
Y_{(i)}(u,\,v,\,a)=Y_{(j)}(u,\,v,\,a).
$$
\end{lm}

\begin{proof}
Firstly, observe that since $u_i=u_{-i}=v_i=v_{-i}=0$, one has
$$Y(e_j,\,u,\,0),\ Y(e_{-j},\,v\eps{j},\,a),\ Y(e_{-j},\,ua\eps{-j},\,0)\in\ps i\cap\ps{-i}.$$ 
Thus, $Y_{(j)}(u,\,v,\,a)$ also belongs to $\ps i\cap\ps{-i}$. Then, 
\begin{multline*}
\!\,^{Y_{(j)}(u,\,v,\,a)}X_{ij}(1)=\!\,^{Y_{(j)}(u,\,v,\,a)}Y(e_i,\,e_{-j}\eps{-j},\,0)=Y(e_i,\,e_{-j}\eps{-j},\,0)=X_{ij}(1),
\end{multline*}
or, what is the same, $[Y_{(j)}(u,\,v,\,a),\,X_{ij}(1)]=1$. Similarly, one can check that $[Y_{(j)}(u,\,v,\,a),\,X_{ji}(-1)]=1$. Thus, $Y_{(j)}(u,\,v,\,a)$ commutes with $W_{ij}(1)$ and using Lemma~\ref{weyl-conjugation} one gets
\begin{multline*}
Y_{(j)}(u,\,v,\,a)=\!\,^{W_{ij}(1)}Y_{(j)}(u,\,v,\,a)=\\=\!\,^{W_{ij}(1)}\big([Y(e_j,\,u,\,0),\,Y(e_{-j},\,v\eps{j},\,a)]Y(e_{-j},\,ua\eps{-j},\,0)\big)=\\=[Y(e_i,\,u,\,0),\,Y(e_{-i},\,v\eps{i},\,a)]Y(e_{-i},\,ua\eps{-i},\,0)=Y_{(i)}(u,\,v,\,a).
\end{multline*}
\end{proof}

\begin{rk}
For $u$ and $v$ having only one pair of zero coordinates we have not yet proven that  $Y_{(i)}(u,\,v,\,a)=Y_{(-i)}(u,\,v,\,a)$. This is why for the time being we have to keep the index in the notation for our ESD-generators. 
\end{rk}

\begin{df}
Define the ({\it Steinberg}) {\it Levi subgroup} $\ls i=\ps i\cap\ps{-i}$.
\end{df}

\begin{rk}
Observe that for $g\in\ls i$ one has $\phi(g)e_i=e_i$ and $\phi(g)e_{-i}=e_{-i}$. Indeed, take $h\neq k$, then the first equality holds for $g=X_{kh}(a)$ with $\{-k,h\}\not\ni i$ and the second one for $g=X_{kh}(a)$ with $\{-k,h\}\not\ni -i$.
\end{rk}

\begin{lm}
For $u$, $v\in V$ such that $\lan u,\,v\ran=0$, $u_i=u_{-i}=v_i=v_{-i}=0$, $a\in R$, $g\in\ls i$, one has
$$
g\,Y_{(i)}(u,\,v,\,a)g\inv=Y_{(i)}(\phi(g)u,\,\phi(g)v,\,a).
$$
\end{lm}

\begin{rk}
Since $\lan\phi(g)u,\,e_i\ran=\lan\phi(g)u,\,\phi(g)e_i\ran=\lan u,\,e_i\ran=0$, one can see that $(\phi(g)u)_{-i}=0$ and similarly $(\phi(g)u)_{i}=(\phi(g)v)_{-i}=(\phi(g)v)_{i}=0$. Thus, $Y_{(i)}(\phi(g)u,\,\phi(g)v,\,a)$ is well-defined.
\end{rk}

\begin{proof}
Using $g\in\ps i\cap\ps{-i}$ and Lemma~\ref{y-conjugated-by-ps} one gets
\begin{multline*}
\!\,^gY_{(i)}(u,\,v,\,a)=\!\,^g\big([Y(e_i,\,u,\,0),\,Y(e_{-i},\,v\eps i,\,a)]Y(e_{-i},\,ua\eps{-i},\,0)\big)=\\=[Y(e_i,\,\phi(g)u,\,0),\,Y(e_{-i},\,\phi(g)v\eps i,\,a)]Y(e_{-i},\,\phi(g)ua\eps{-i},\,0)=\\=Y_{(i)}(\phi(g)u,\,\phi(g)v,\,a).
\end{multline*}
\end{proof}

\begin{rk}
Lemma~\ref{unipotent-decomposition} implies that for $v$ such that $v_{-i}=v_j=v_{-j}=0$, one has $Y(e_i,\,v,\,a)\in\ls j$. 
\end{rk}

\begin{rk}
For $w$ orthogonal to both $u$ and $v$, $\lan u,\,w\ran=\lan v,\,w\ran=0$, one has $T(u,\,v,\,a)w=w$. Below, in the computations this fact is frequently used without any special reference. 
\end{rk}

\begin{lm}
\label{z-additivity}
For $u$, $v\in V$ such that $u_i=u_{-i}=u_j=u_{-j}=0$, $v_i=v_{-i}=0$ and $\lan u,\,v\ran=0$, and for $a$, $b\in R$, one has
$$
Y_{(i)}(u,\,v,\,a)Y_{(i)}(u,\,e_jb,\,0)=Y_{(i)}(u,\,v+e_jb,\,a+v_{-j}b\eps{-j}).
$$
\end{lm}

\begin{proof}
Start with the right-hand side
\begin{multline*}
Y_{(i)}(u,\,v+e_jb,\,a+v_{-j}b\eps{-j})=\\
=[Y(e_i,\,u,\,0),\,Y(e_{-i},\,(v+e_jb)\eps{i},\,a+v_{-j}b\eps{-j})]Y(e_{-i},\,u(a+v_{-j}b\eps{-j})\eps{-i},\,0).
\end{multline*}
Decompose $Y(e_{-i},\,(v+e_jb)\eps{i},\,a+v_{-j}b\eps{-j})$ inside the commutator and use the familiar identity $[a,\,bc]=[a,\,b]\cdot\,^b[a,\,c]$ to obtain
\begin{multline*}
[Y(e_i,\,u,\,0),\,Y(e_{-i},\,(v+e_jb)\eps{i},\,a+v_{-j}b\eps{-j})]=\\
=[Y(e_i,\,u,\,0),\,Y(e_{-i},\,v\eps{i},\,a)Y(e_{-i},\,e_jb\eps{i},\,0)]=\\
=[Y(e_i,\,u,\,0),\,Y(e_{-i},\,v\eps{i},\,a)]\cdot\,^{Y(e_{-i},\,v\eps{i},\,a)}Y(e_j,\,ub,\,0).
\end{multline*}
Observe that in general $Y(e_{-i},\,v\eps{i},\,a)$ does not lie in $\ps j$, but $Y(e_j,\,ub,\,0)$ always lies in $\ps{-i}$. So that we can compute the conjugate as follows
\begin{multline*}
\!\,^{Y(e_{-i},\,v\eps{i},\,a)}Y(e_j,\,ub,\,0)=\\=Y(e_j,\,ub,\,0)[Y(e_j,\,-ub,\,0),\,Y(e_{-i},\,v\eps{i},\,a)]=\\
=Y(e_j,\,ub,\,0)Y(e_{-i},\,-ubv_{-j}\eps i\eps j,\,0).
\end{multline*}
Thus,
\begin{multline*}
Y_{(i)}(u,\,v+e_jb,\,a+v_{-j}b\eps{-j})=\\
=[Y(e_i,\,u,\,0),\,Y(e_{-i},\,v\eps{i},\,a)]Y(e_j,\,ub,\,0)Y(e_{-i},\,ua\eps{-i},\,0).
\end{multline*}
Finally, it remains to observe that $Y(e_j,\,ub,\,0)\in\ps{-i}$ commutes with $Y(e_{-i},\,ua\eps{-i},\,0)$ and
$$
Y_{(i)}(u,\,v+e_jb,\,a+v_{-j}b\eps j)
=Y_{(i)}(u,\,v,\,a)Y_{(i)}(u,\,e_jb,\,0).
$$
\end{proof}

We have defined $Y_{(i)}(u,\,v,\,a)$ only for $u$ and $v$ both having pairs of zeros. Now we want to define the generators for an arbitrary $v$.

\begin{df}
For $u$, $v\in V$ such that $u_i=u_{-i}=0$, $\lan u,\,v\ran=0$, and $a\in R$ define
\begin{multline*}
Y_{(i)}(u,\,v,\,a)=\\=Y_{(i)}(u,v-e_iv_i-e_{-i}v_{-i},a-v_iv_{-i}\eps i)Y(e_i,\,uv_i,\,0)Y(e_{-i},\,uv_{-i},\,0).
\end{multline*}
\end{df}

\begin{rk}
Observe that the above definition coincides with the old one for $v$ with $v_i=v_{-i}=0$, so that we can use the same notation for the generator. Observe also that the right-hand side is well-defined. Namely, $v-e_iv_i-e_{-i}v_{-i}$ has zeros in positions $\pm i$ and is orthogonal to $u$. Indeed, it is obvious since $v$, $e_i$ and $e_{-i}$ are orthogonal to $u$.
\end{rk}

\begin{lm}
\label{w-conjugation}
For $g\in\ls i$, $u$, $v\in V$ such that $u_i=u_{-i}=0$, $\lan u,\,v\ran=0$, and $a\in R$ one has
$$
g\,Y_{(i)}(u,\,v,\,a)g\inv=Y_{(i)}(\phi(g)u,\,\phi(g)v,\,a).
$$
\end{lm}

\begin{proof}
Since $g\in\ls i$ one gets 
$$
\big(\phi(g)v\big)_i=\lan\phi(g)v,\,e_{-i}\ran\eps i=\lan\phi(g)v,\,\phi(g)e_{-i}\ran\eps i=\lan v,\,e_{-i}\ran\eps i=v_i
$$
and similarly $\big(\phi(g)v\big)_{-i}=v_{-i}$. Then
\begin{multline*}
\!\,^gY_{(i)}(u,\,v,\,a)=\\
=\!\,^gY_{(i)}(u,v-e_iv_i-e_{-i}v_{-i},a-v_iv_{-i}\eps i)\cdot\,^gY(e_i,\,uv_i,\,0)\cdot\,^gY(e_{-i},\,uv_{-i},\,0)=\\
=Y_{(i)}(\phi(g)u,\phi(g)v-e_i\big(\phi(g)v\big)_i-e_{-i}\big(\phi(g)v\big)_{-i},a-\big(\phi(g)v\big)_i\big(\phi(g)v\big)_{-i}\eps i)\cdot\\
\cdot Y(e_i,\,\phi(g)u\big(\phi(g)v\big)_i,\,0)Y(e_{-i},\,\phi(g)u\big(\phi(g)v\big)_{-i},\,0)=\\
=Y_{(i)}(\phi(g)u,\,\phi(g)v,\,a).
\end{multline*}
\end{proof}

\begin{rk}
Obviously, for $u$, $v\in V$ such that $u_i=u_{-i}=u_j=u_{-j}=0$ and $v_j=v_{-j}=0$, $a\in R$, one has $Y_{(i)}(u,\,v,\,a)\in\ls j$. Indeed, it is a product of elements from $\ls j$ by definition.
\end{rk}

\begin{lm}
\label{w-correctness}
For $u$, $v\in V$ such that $u_i=u_{-i}=u_j=u_{-j}=0$, $\lan u,\,v\ran=0$, and $a\in R$ one has
$$
Y_{(i)}(u,\,v,\,a)=Y_{(j)}(u,\,v,\,a).
$$
\end{lm}

\begin{proof}
Denote $\tilde v=v-e_iv_i-e_{-i}v_{-i}$, $\tilde a=a-v_iv_{-i}\eps i$. Then by Lemma~\ref{z-additivity} one gets
\begin{multline*}
Y_{(i)}(u,\,\tilde v,\,\tilde a)=Y_{(i)}(u,\,\tilde v-e_{-j}v_{-j},\,\tilde a-v_jv_{-j}\eps j)Y_{(i)}(u,\,e_{-j}v_{-j},\,0)=\\
=Y_{(i)}(u,\,\tilde v-e_{-j}v_{-j}-e_jv_j,\,\tilde a-v_jv_{-j}\eps j)Y_{(i)}(u,\,e_{j}v_{j},\,0)Y_{(i)}(u,\,e_{-j}v_{-j},\,0).
\end{multline*}
Further, denote $\tilde{\tilde v}=\tilde v-e_jv_j-e_{-j}v_{-j}$ and $\tilde{\tilde a}=\tilde a-v_jv_{-j}\eps j$. Then one has
\begin{multline*}
Y_{(i)}(u,\,v,\,a)=Y_{(i)}(u,\,\tilde{\tilde v},\,\tilde{\tilde a})Y(e_j,\,uv_j,\,0)Y(e_{-j},\,uv_{-j},\,0)\cdot\\
\cdot Y(e_i,\,uv_i,\,0)Y(e_{-i},\,uv_{-i},\,0).
\end{multline*}
Changing roles of $i$ and $j$ one gets
\begin{multline*}
Y_{(j)}(u,\,v,\,a)=Y_{(j)}(u,\,\tilde{\tilde v},\,\tilde{\tilde a})Y(e_i,\,uv_i,\,0)Y(e_{-i},\,uv_{-i},\,0)\cdot\\
\cdot Y(e_j,\,uv_j,\,0)Y(e_{-j},\,uv_{-j},\,0).
\end{multline*}
But $Y_{(i)}(u,\,\tilde{\tilde v},\,\tilde{\tilde a})=Y_{(j)}(u,\,\tilde{\tilde v},\,\tilde{\tilde a})$ by Lemma~\ref{z-correctness}. Finally, it remains to observe that $Y(e_{-i},\,uv_{-i},\,0)$ and $Y(e_i,\,uv_i,\,0)$ commute with both $Y(e_j,\,uv_j,\,0)$ and $Y(e_{-j},\,uv_{-j},\,0)$. This is obvious from the fact that the above elements lie in $\ls j$.
\end{proof}

\begin{rk}
For $u$ equal to the base vector $e_j$ using Lemma~\ref{z-decomposition-for-y} one gets 
\begin{multline*}
Y_{(i)}(e_j,\,v,\,a)=\\=Y(e_j,v-e_iv_i-e_{-i}v_{-i},a-v_iv_{-i}\eps i)Y(e_i,\,e_jv_i,\,0)Y(e_{-i},\,e_jv_{-i},\,0)=\\
=Y(e_j,\,v,\,a).
\end{multline*}
\end{rk}

\begin{df}
For $u$ having at least two pairs of zeros the element $Y_{(i)}(u,\,v,\,a)$ does not depend on the choice of $i$ by Lemma~\ref{w-correctness}. In this situation we will often omit the index in the notation,
$$
Y(u,\,v,\,a)=Y_{(i)}(u,\,v,\,a).
$$
\end{df}

\section{Long root type elements}

\begin{df}
For $u\in V$ and $a\in R$ define
\begin{multline*}
X_{(i)}(u,\,0,\,a)=Y_{(i)}(u-e_iu_i-e_{-i}u_{-i},\,0,\,a)Y(e_iu_i+e_{-i}u_{-i},\,0,\,a)\cdot\\
\cdot Y(e_iu_i+e_{-i}u_{-i},\,(u-e_iu_i-e_{-i}u_{-i})a,\,0).
\end{multline*}
\end{df}

Our objective in this section is to show that $X_{(i)}(u,\,0,\,a)$ does not depend on the choice of $i$.

\begin{lm}
\label{long-add}
For $u$, $v\in V$ such that $u_i=u_{-i}=u_j=u_{-j}=0$, $v_{i}=v_{-i}=0$, $\lan u,\,v\ran=0$, $a$, $b\in R$ one has
$$
Y(u,\,v,\,a+b)=Y(u,\,v,\,a)Y(u,\,0,\,b).
$$
\end{lm}

\begin{proof}
Decompose $Y(u,\,v,\,a+b)$ as a product of unipotents
\begin{multline*}
Y(u,\,v,\,a+b)
=[Y(e_i,\,u,\,0),\,Y(e_{-i},\,v\eps i,\,a+b)]Y(e_{-i},\,u(a+b)\eps{-i},\,0).
\end{multline*}
Using $[a,\,bc]=[a,\,b]\cdot\,^b[a,\,c]$ we obtain
\begin{multline*}
[Y(e_i,\,u,\,0),\,Y(e_{-i},\,v\eps i,\,a+b)]=\\
=[Y(e_i,\,u,\,0),\,Y(e_{-i},\,v\eps i,\,a)]\cdot\,^{Y(e_{-i},\,v\eps i,\,a)}[Y(e_i,\,u,\,0),\,Y(e_{-i},\,0,\,b)].
\end{multline*}
Since $\lan u,\,v\ran=0$ one has 
$$\!\,^{Y(e_{-i},\,v\eps i,\,a)}Y(e_{-i},\,ub\eps{-i},\,0)=Y(e_{-i},\,ub\eps{-i},\,0),$$ 
and thus
\begin{multline*}
Y(u,\,v,\,a+b)=\\
=[Y(e_i,\,u,\,0),\,Y(e_{-i},\,v\eps i,\,a+b)]Y(e_{-i},\,ub\eps{-i},\,0)Y(e_{-i},\,ua\eps{-i},\,0)=\\
=[Y(e_i,\,u,\,0),\,Y(e_{-i},\,v\eps i,\,a)]\cdot\,^{Y(e_{-i},\,v\eps i,\,a)}[Y(e_i,\,u,\,0),\,Y(e_{-i},\,0,\,b)]\cdot\\
\cdot\,^{Y(e_{-i},\,v\eps i,\,a)}Y(e_{-i},\,ub\eps{-i},\,0)\cdot Y(e_{-i},\,ua\eps{-i},\,0)=\\
=[Y(e_i,\,u,\,0),\,Y(e_{-i},\,v\eps i,\,0)]\cdot\,^{Y(e_{-i},\,v\eps i,\,a)}Y_{(i)}(u,\,0,\,b)\cdot Y(e_{-i},\,ua\eps{-i},\,0).
\end{multline*}
Recall that $Y_{(j)}(u,\,0,\,b)\in\ls i$ commutes with both $Y(e_{-i},\,v\eps i,\,0)$ and $Y(e_{-i},\,ua\eps{-i},\,0)$, so that
\begin{multline*}
Y(u,\,v,\,a+b)=[Y(e_i,\,u,\,0),\,Y(e_{-i},\,v\eps i,\,a)]Y(e_{-i},\,ua\eps{-i},\,0)Y(u,\,0,\,b)=\\
=Y_{(i)}(u,\,v,\,a)Y(u,\,0,\,b).
\end{multline*}
\end{proof}

\begin{rk}
For $u\in V$ having at least two pairs of zeros one has $Y(u,\,0,\,0)=1$ and $Y(u,\,0,\,a)\inv=Y(u,\,0,\,-a)$.
\end{rk}

\begin{lm}
\label{w=zz}
For $u$, $v\in V$ such that $u_i=u_{-i}=u_j=u_{-j}=0$, $\lan u,\,v\ran=0$, and $a\in R$ one has
$$
Y(u,\,v,\,a)=Y(u,\,v-e_iv_i-e_{-i}v_{-i},\,a)Y(u,\,e_iv_i+e_{-i}v_{-i},\,0).
$$
\end{lm}

\begin{proof}
By definition
\begin{multline*}
Y(u,\,v,\,a)=Y_{(i)}(u,\,v,\,a)=\\
=Y_{(i)}(u,v-e_iv_i-e_{-i}v_{-i},a-v_iv_{-i}\eps i)Y(e_i,\,uv_i,\,0)Y(e_{-i},\,uv_{-i},\,0).
\end{multline*}
Denote $\tilde v=v-e_iv_i-e_{-i}v_{-i}$, then by the previous lemma
$$
Y(u,\tilde v,a-v_iv_{-i}\eps i)=Y(u,\,\tilde v,\,a)Y(u,\,0,\,-v_iv_{-i}\eps i),
$$
and thus
\begin{multline*}
Y(u,\,v,\,a)=Y(u,\,\tilde v,\,a)Y(u,\,0,\,-v_iv_{-i}\eps i)Y(e_i,\,uv_i,\,0)Y(e_{-i},\,uv_{-i},\,0)=\\
=Y(u,\,\tilde v,\,a)Y_{(i)}(u,\,e_iv_i+e_{-i}v_{-i},\,0).
\end{multline*}
\end{proof}

\begin{cl*}
Consider $u\in V$, $a\in R$, and denote $v=e_iu_i+e_{-i}u_{-i}$, $v'=e_ju_j+e_{-j}u_{-j}$, $\tilde u=u-v$, $\tilde{\tilde u}=\tilde u-v'$. Then one has
\begin{multline*}
X_{(i)}(u,\,0,\,a)=Y_{(i)}(\tilde u,\,0,\,a)Y(v,\,0,\,a)Y(v,\,\tilde ua,\,0)=\\
=Y_{(i)}(\tilde u,\,0,\,a)Y(v,\,0,\,a)Y(v,\,\tilde{\tilde u}a,\,0)Y(v,\,v'a,\,0).
\end{multline*}
\end{cl*}

\begin{proof}
Since $l\geq3$ the vector $v$ has at least two pairs of zeros, so that one can apply the previous lemma.
\end{proof}

\begin{lm}
\label{short-is-three-long}
Consider $j$, $k\not\in\{\pm i\}$ but not necessarily distinct or having non-zero sum, $u$, $v\in V$ such that $u_i=u_{-i}=u_j=u_{-j}=0$ and $v_i=v_{-i}=v_k=v_{-k}=0$, $\lan u,\,v\ran=0$, and $a\in R$. Then holds
$$
Y_{(i)}(u+v,\,0,\,a)=Y(u,\,0,\,a)Y(v,\,0,\,a)Y(v,\,ua,\,0).
$$
\end{lm}

\begin{proof}
Decomposing $Y(e_i,\,u+v,\,0)=Y(e_i,\,v,\,0)Y(e_i,\,u,\,0)$ inside the commutator and using $[ab,\,c]=\,^a[b,\,c]\cdot[a,\,c]$ we get
\begin{multline*}
Y_{(i)}(u+v,\,0,\,a)=[Y(e_i,\,u+v,\,0),\,Y(e_{-i},\,0,\,a)]Y(e_{-i},\,(u+v)a\eps{-i},\,0)=\\
=\!\,^{Y(e_i,\,v,\,0)}[Y(e_i,\,u,\,0),\,Y(e_{-i},\,0,\,a)]\cdot[Y(e_i,\,v,\,0),\,Y(e_{-i},\,0,\,a)]\cdot\\
\cdot Y(e_{-i},\,va\eps{-i},\,0)Y(e_{-i},\,ua\eps{-i},\,0)=\!\,^{Y(e_i,\,v,\,0)}[Y(e_i,\,u,\,0),\,Y(e_{-i},\,0,\,a)]\cdot\\
\cdot Y_{(i)}(v,\,0,\,a)Y(e_{-i},\,ua\eps{-i},\,0).
\end{multline*}
Observing that $Y_{(k)}(v,\,0,\,a)\in\ls i$ commutes with both $Y(e_i,\,v,\,0)$ and $Y(e_{-i},\,-ua\eps{-i},\,0)$, we get
\begin{multline*}
\!\,^{Y(e_i,\,v,\,0)}[Y(e_i,\,u,\,0),\,Y(e_{-i},\,0,\,a)]\cdot Y(v,\,0,\,a)\cdot Y(e_{-i},\,ua\eps{-i},\,0)=\\
=\!\,^{Y(e_i,\,v,\,0)}[Y(e_i,\,u,\,0),\,Y(e_{-i},\,0,\,a)]\cdot\,^{Y(e_i,\,v,\,0)}Y(e_{-i},\,ua\eps{-i},\,0)\cdot\\
\cdot\,^{Y(e_i,\,v,\,0)}Y(e_{-i},\,-ua\eps{-i},\,0)\cdot Y(v,\,0,\,a)\cdot Y(e_{-i},\,ua\eps{-i},\,0)=\\
=\!\,^{Y(e_i,\,v,\,0)}Y_{(i)}(u,\,0,\,a)\cdot Y(v,\,0,\,a)\cdot[Y(e_i,\,v,\,0),\,Y(e_{-i},\,-ua\eps{-i},\,0)]=\\
=\!\,^{Y(e_i,\,v,\,0)}Y(u,\,0,\,a)\cdot Y(v,\,0,\,a)\cdot Y(v,\,ua,\,0).
\end{multline*}
Finally, use that $Y_{(j)}(u,\,0,\,a)\in\ls i$ commutes with $Y(e_i,\,v,\,0)$.
\end{proof}

\begin{cl*}
Consider $u\in V$, $a\in R$, and denote $v=e_iu_i+e_{-i}u_{-i}$, $v'=e_ju_j+e_{-j}u_{-j}$, $\tilde u=u-v$, $\tilde{\tilde u}=\tilde u-v'$. Then
\begin{multline*}
X_{(i)}(u,\,0,\,a)=Y_{(i)}(\tilde u,\,0,\,a)Y(v,\,0,\,a)Y(v,\,\tilde{\tilde u}a,\,0)Y(v,\,v'a,\,0)=\\
=Y(\tilde{\tilde u},\,0,\,a)Y(v',\,0,\,a)Y(v',\,\tilde{\tilde u}a,\,0)Y(v,\,0,\,a)Y(v,\,\tilde{\tilde u}a,\,0)Y(v,\,v'a,\,0).
\end{multline*}
\end{cl*}

\begin{lm}
\label{short-symmetry}
For $j$, $k\not\in\{\pm i\}$ but not necessarily distinct or having non-zero sum, $u$, $v\in V$ such that $u_i=u_{-i}=u_{j}=u_{-j}=0$ and $v_i=v_{-i}=v_{k}=v_{-k}=0$, $\lan u,\,v\ran=0$, and $a\in R$, one has
$$
Y(u,\,va,\,0)=Y(v,\,ua,\,0).
$$
\end{lm}

\begin{proof}
By Lemma~\ref{short-is-three-long} one has
$$
Y(v,\,ua,\,0)=Y(v,\,0,\,-a)Y(u,\,0,\,-a)Y_{(i)}(u+v,\,0,\,a)
$$
and similarly
$$
Y(u,\,va,\,0)=Y(u,\,0,\,-a)Y(v,\,0,\,-a)Y_{(i)}(v+u,\,0,\,a).
$$
But $Y_{(i)}(u,\,0,\,-a)\in\ls i$ commutes with $Y_{(i)}(v,\,0,\,-a)$.
\end{proof}

\begin{lm}
\label{correctness}
For $u\in V$ and $a\in R$ one has
$$
X_{(i)}(u,\,0,\,a)=X_{(j)}(u,\,0,\,a).
$$
\end{lm}

\begin{proof}
Set 
$$v=e_iu_i+e_{-i}u_{-i},\quad v'=e_ju_j+e_{-j}u_{-j},\quad \tilde u=u-v,\quad \tilde{\tilde u}=\tilde u-v'.$$ 
As we have already noticed, Lemmas~\ref{w=zz} and \ref{short-is-three-long} imply that
\begin{multline*}
X_{(i)}(u,\,0,\,a)=Y_{(i)}(\tilde u,\,0,\,a)Y(v,\,0,\,a)Y(v,\,\tilde ua,\,0)=\\
=Y(\tilde{\tilde u},\,0,\,a)Y(v',\,0,\,a)Y(v',\,\tilde{\tilde u}a,\,0)Y(v,\,0,\,a)Y(v,\,\tilde{\tilde u}a,\,0)Y(v,\,v'a,\,0).
\end{multline*}
Interchanging roles of $i$ and $j$ one has
\begin{multline*}
X_{(j)}(u,\,0,\,a)=\\
=Y(\tilde{\tilde u},\,0,\,a)Y(v,\,0,\,a)Y(v,\,\tilde{\tilde u}a,\,0)Y(v',\,0,\,a)Y(v',\,\tilde{\tilde u}a,\,0)Y(v',\,va,\,0).
\end{multline*}
One has $Y(v,\,v'a,\,0)=Y(v',\,va,\,0)$ by Lemma~\ref{short-symmetry}. Now we have to show that $Y(v,\,0,\,a)Y(v,\,\tilde{\tilde u}a,\,0)$ commutes with $Y(v',\,0,\,a)Y(v',\,\tilde{\tilde u}a,\,0)$. With this end fix $k\not\in\{\pm i,\pm j\}$. Observe that $Y_{(i)}(v',\,0,\,a)\in\ls k$ commutes with $Y_{(k)}(v,\,\tilde{\tilde u}a,\,0)$ by Lemma~\ref{w-conjugation}. Next, by Lemma~\ref{short-symmetry} one has
$$
Y(v,\,\tilde{\tilde u}a,\,0)=Y(\tilde{\tilde u},\,va,\,0)
$$
and
$$
Y(v',\,\tilde{\tilde u}a,\,0)=Y(\tilde{\tilde u},\,v'a,\,0).
$$
Now, Lemma~\ref{w=zz} implies that
\begin{multline*}
Y(v,\,\tilde{\tilde u}a,\,0)Y(v',\,\tilde{\tilde u}a,\,0)=Y(\tilde{\tilde u},\,va,\,0)Y(\tilde{\tilde u},\,v'a,\,0)=\\
=Y(\tilde{\tilde u},\,va+v'a,\,0)=Y(\tilde{\tilde u},\,v'a,\,0)Y(\tilde{\tilde u},\,va,\,0)=\\
=Y(v',\,\tilde{\tilde u}a,\,0)Y(v,\,\tilde{\tilde u}a,\,0).
\end{multline*}
Finally, it remains to observe that $Y_{(j)}(v,\,0,\,a)\in\ls k$ commutes with both $Y_{(k)}(v',\,0,\,a)$ and $Y_{(k)}(v',\,\tilde{\tilde u}a,\,0)$ by Lemma~\ref{w-conjugation}.
\end{proof}

\begin{rk}
Since $X_{(i)}(u,\,0,\,a)$ does not depend on the choice of $i$ we will often omit the index in the notation
$$
X(u,\,0,\,a)=X_{(i)}(u,\,0,\,a).
$$
\end{rk}

\section{Proof of the Main Theorem}

Our objective in this section is to prove the conjugation formula for long-root type generators,
$$gX(u,\,0,\,a)g\inv=X(\phi(g)u,\,0,\,a),$$ 
and to complete our proof of the Main Theorem.

\begin{lm}
\label{conjugation-by-long}
For any $u\in V$, any $a$, $b\in R$ and any index $i$, one has
$$
\!\,^{X_{i,-i}(b)}X(u,\,0,\,a)=X(T_{i,-i}(b)u,\,0,\,a).
$$
\end{lm}

\begin{proof}
Choosing $j\not\in\{\pm i\}$ and decomposing $X(u,\,0,\,a)$ with the use of Lemmas~\ref{w=zz} and \ref{short-is-three-long}, one has
\begin{multline*}
X_{(i)}(u,\,0,\,a)=\\
=Y(\tilde{\tilde u},\,0,\,a)Y(v',\,0,\,a)Y(v',\,\tilde{\tilde u}a,\,0)Y(v,\,0,\,a)Y(v,\,\tilde{\tilde u}a,\,0)Y(v,\,v'a,\,0),
\end{multline*}
where $v=e_iu_i+e_{-i}u_{-i}$, $v'=e_ju_j+e_{-j}u_{-j}$, $\tilde{\tilde u}=u-v-v'$. Observe that 
$$
e_i\big(T_{i,-i}(b)u\big)_i+e_{-i}\big(T_{i,-i}(b)u\big)_{-i}=T_{i,-i}(b)v
$$
and, obviously,
$$
e_j\big(T_{i,-i}(b)u\big)_j+e_{-j}\big(T_{i,-i}(b)u\big)_{-j}=v'=T_{i,-i}(b)v'
$$
and similarly
$$
T_{i,-i}(b)u-T_{i,-i}(b)v-v'=\tilde{\tilde u}=T_{i,-i}(b)\tilde{\tilde u}.
$$
Thus, we only have to check the conjugation formula for the factors of $X(u,\,0,\,a)$ in the above decomposition. Indeed, $X_{i,-i}(b)\in\ls j\cap\ls k$ for any $k\not\in\{\pm i,\pm j\}$, and any factor is equal either to $Y_{(j)}(\hat u,\,\hat v,\,\hat a)$, or to $Y_{(k)}(\hat u,\,\hat v,\,\hat a)$ for some $\hat u,\,\hat v\in V$, $\hat a\in R$. 
\end{proof}

\begin{lm}
\label{lm0}
For $u$, $v$, $w\in V$ such that 
$$u_i=u_{-i}=u_j=u_{-j}=0,\qquad v_j=v_{-j}=0,\qquad w_j=w_{-j}=0,$$

and $\lan u,\,v\ran=0$, $\lan u,\,w\ran=0$, $\lan v,\,w\ran=0$, one has
$$
Y(u,\,v,\,0)Y(u,\,w,\,0)=Y(u,\,v+w,\,0).
$$
\end{lm}

\begin{proof}
Firstly, using $[a,\,bc]=[a,\,b]\cdot\!\,^b[a,\,c]$ we obtain
\begin{multline*}
Y(u,\,v+w,\,0)
=[Y(e_j,\,u,\,0),\,Y(e_{-j},\,(v+w)\eps j,\,0)]=\\
=Y(u,\,v,\,0)\cdot\,^{Y(e_{-j},\,v\eps j,\,0)}Y(u,\,w,\,0).
\end{multline*}
Now, observe that $Y_{(i)}(u,\,w,\,0)\in\ps{-j}$ commutes with $Y(e_{-j},\,v\eps j,\,0)$.
\end{proof}

\begin{lm}
\label{w-symmetry}
Let $v$, $v'\in V$ be vectors having only $\pm i$-th and $\pm j$-th non-zero coordinates respectively; consider also $v''\in V$ such that $(v'')_i=(v'')_{-i}=(v'')_j=(v'')_{-j}=0$. Set $w=v'+v''$. Then
$$Y(v'',\,v,\,0)Y(v',\,v,\,0)=Y_{(i)}(w,\,v,\,0).$$
\end{lm}

\begin{rk}
We use this lemma only for the case where $v''$ has only $\pm k$-th non-zero coordinates. If $w$ has at least 2 pairs of non-zero coordinates, the claim is an obvious consequence of Lemmas~\ref{short-symmetry} and \ref{w=zz}. This means that both Lemmas~\ref{w-symmetry} and \ref{lm0} are only required in the case $l=3$ and are not relevant when $l\geq4$.
\end{rk}

\begin{proof}
Using Lemma~\ref{short-is-three-long} one has
\begin{multline*}
Y_{(i)}(w,\,v,\,0)=Y_{(i)}(w,\,0,\,-v_iv_{-i}\eps i)Y(e_i,\,wv_i,\,0)Y(e_{-i},\,wv_{-i},\,0)=\\
=Y(v'',\,0,\,-v_iv_{-i}\eps i)Y(v',\,0,\,-v_iv_{-i}\eps i)Y(v'',\,-v'v_iv_{-i}\eps i,\,0)\cdot\\
\cdot Y(e_i,\,v''v_i,\,0)Y(e_i,\,v'v_i,\,0)Y(e_{-i},\,v''v_{-i},\,0)Y(e_{-i},\,v'v_{-i},\,0).
\end{multline*}
We want to change the order of factors in the above product to obtain $Y(v'',\,v,,0)$, more precisely, we want to put $Y(e_i,\,v''v_i,\,0)$ and $Y(e_{-i},\,v''v_{-i},\,0)$ right after the first factor. Since $Y_{(k)}(v',\,0,\,-v_iv_{-i}\eps i)$ and $Y_{(j)}(v'',\,-v'v_iv_{-i}\eps i,\,0)$ lie in $\ls i$, this elements commute with both $Y(e_i,\,v''v_i,\,0)$ and $Y(e_{-i},\,v''v_{-i},\,0)$. But $Y(e_i,\,v'v_i,\,0)$ does not commute with $Y(e_{-i},\,v''v_{-i},\,0)$, thus we obtain the following.
%
\begin{multline*}
Y_{(i)}(w,\,v,\,0)=\\
=Y(v'',\,v,\,0)Y(v',\,0,\,-v_iv_{-i}\eps i)Y(v'',\,-v'v_iv_{-i}\eps i,\,0)Y(e_i,\,v'v_i,\,0)\cdot\\
\cdot[Y(e_i,\,-v'v_i,\,0),\,Y(e_{-i},\,-v''v_{-i},\,0)]Y(e_{-i},\,v'v_{-i},\,0)
\end{multline*}
Observe that 
\begin{multline*}
[Y(e_i,\,-v'v_i,\,0),\,Y(e_{-i},\,-v''v_{-i},\,0)]=\\=Y(-v'v_i,\,-v''v_{-i}\eps i,\,0)=Y(v'',\,v'v_iv_{-i}\eps i,\,0)\in\ls i
\end{multline*}
by Lemma~\ref{short-symmetry} and thus this element commutes with $Y(e_i,\,v'v_i,\,0)$. Further, by Lemma~\ref{lm0} one has
$$
Y(v'',\,-v'v_iv_{-i}\eps i,\,0)Y(v'',\,v'v_iv_{-i}\eps i,\,0)=Y(v'',\,0,\,0)=1.
$$
Plugging this into the above formula, we get the claim.
\end{proof}

\begin{lm}
\label{conjugation-by-short}
For any $j\not\in\{\pm k\}$, any $u\in V$ and any $a$, $b\in R$, one has
$$
\!\,^{X_{jk}(b)}X(u,\,0,\,a)=X(T_{jk}(b)u,\,0,\,a).
$$
\end{lm}

\begin{proof}
Fix $i\not\in\{\pm j,\pm k\}$.  Combining Lemmas~\ref{w=zz}, \ref{short-symmetry} and \ref{w-symmetry} we get
\begin{multline*}
X_{(i)}(u,\,0,\,a)=\\
=Y_{(i)}(\tilde u,\,0,\,a)Y(v,\,0,\,a)Y(v,\,\tilde{\tilde{\tilde u}}a,\,0)Y(v,\,v''a,\,0)Y(v,\,v'a,\,0)=\\
=Y_{(i)}(\tilde u,\,0,\,a)Y(v,\,0,\,a)Y(v,\,\tilde{\tilde{\tilde u}}a,\,0)Y_{(i)}(w,\,va,\,0),
\end{multline*}
where 
$$v=e_iu_i+e_{-i}u_{-i},\qquad\tilde u=u-v,\qquad v'=e_ju_j+e_{-j}u_{-j},\qquad\tilde{\tilde u}=\tilde u-v',$$
$$v''=e_ku_k+e_{-k}u_{-k},\qquad\tilde{\tilde{\tilde u}}=\tilde{\tilde u}-v'',\qquad w=v'+v''.$$
Observe that 
$$
e_i\big(T_{jk}(b)u\big)_i+e_{-i}\big(T_{jk}(b)u\big)_{-i}=v=T_{jk}(b)v,
$$
$$
e_j\big(T_{jk}(b)u\big)_j+e_{-j}\big(T_{jk}(b)u\big)_{-j}+e_k\big(T_{jk}(b)u\big)_k+e_{-k}\big(T_{jk}(b)u\big)_{-k}=T_{jk}(b)w,
$$
$$
T_{jk}(b)u-v=T_{jk}(b)\tilde u\quad\text{and}\quad T_{jk}(b)u-v-T_{jk}(b)w=T_{jk}(b)\tilde{\tilde{\tilde u}}.
$$
Since $Y_{(j)}(v,\,0,\,a)$ and $Y_{(j)}(v,\,\tilde{\tilde{\tilde u}}a,\,0)$ lie in $\ls k$ they commute with $X_{jk}(b)=Y(e_{-k},\,-e_ja\eps k,\,0)$. Now, using $X_{jk}(b)\in\ls i$ we obtain the claim.
\end{proof}

\begin{cl*}
Lemmas~{\rm \ref{conjugation-by-long}} and {\rm \ref{conjugation-by-short}} imply that for any $g\in\St\!\Sp(2l,\,R)$, one has
$$
gX(u,\,0,\,a)g\inv=X(\phi(g)u,\,0,\,a).
$$
\end{cl*}

\begin{lm}
\label{generating}
The set of elements $\{X(u,\,0,\,a)\mid u\in V,\ a\in R\}$ generates $\St\!\Sp(2l,\,R)$ as a group.
\end{lm}

\begin{proof}
Firstly, choosing some $i$ and $j$ such that $\Card\{\pm i,\pm j\}=4$ one has
$$
X_{(j)}(e_i,\,0,\,a)=Y(e_i,\,0,\,a)Y(0,\,0,\,a)Y(0,\,e_ia,\,0)=Y(e_i,\,0,\,a)=X_{i,-i}(a).
$$
Now, choosing $k\not\in\{\pm i,\pm j\}$, taking $u=e_{-k}$, $v=-e_j\eps k$ and any $a\in R$, and using Lemma~\ref{short-is-three-long} we obtain
\begin{multline*}
X_{jk}(a)=Y(u,\,va,\,0)=\\
=Y(v,\,0,\,-a)Y(u,\,0,\,-a)Y_{(i)}(u+v,\,0,\,a)=\\
=X_{(i)}(v,\,0,\,-a)X_{(i)}(u,\,0,\,-a)X_{(i)}(u+v,\,0,\,a).
\end{multline*}
\end{proof}

\begin{proof}[Proof of the Main Theorem]
The Main Theorem asserts that
$$\Ker\phi\subseteq\Cent\St\!\Sp(2l,\,R).$$
Indeed, any $g\in\Ker\phi$ commutes with the long-type transvections $X(u,\,0,\,a)$ by Lemmas~\ref{conjugation-by-long} and \ref{conjugation-by-short}, so that it is central by Lemma~\ref{generating}.
\end{proof}

\section{Relations among ESD-generators}

In the two remaining sections we establish another presentation for the symplectic Steinberg group, similar to the one obtained by van der Kallen in the linear case~\cite{vdK}.

\begin{lm}
\label{long-additivity}
For $u\in V$, $a$, $b\in R$ one has
$$
X(u,\,0,\,a)X(u,\,0,\,b)=X(u,\,0,\,a+b).
$$
\end{lm}

\begin{proof}
Choosing $j\not\in\{\pm i\}$ and decomposing $X(u,\,0,\,a)$ with the use of Lemmas~\ref{w=zz}, \ref{short-is-three-long} and \ref{short-symmetry}, one has
\begin{multline*}
X(u,\,0,\,a)=\\
=Y(\tilde{\tilde u},\,0,\,a)Y(v',\,0,\,a)Y(\tilde{\tilde u},\,v'a,\,0)Y(v,\,0,\,a)Y(\tilde{\tilde u},\,va,\,0)Y(v,\,v'a,\,0).
\end{multline*}
Since $\tilde{\tilde u}$, $v$ and $v'$ are orthogonal to $\tilde{\tilde u}$, the vector $u$ is also orthogonal to $\tilde{\tilde u}$, and thus 
$$
\!\,^{Y(\tilde{\tilde u},\,0,\,b)}X(u,\,0,\,a)=X(T(\tilde{\tilde u},\,0,\,b)u,\,0,\,a)=X(u,\,0,\,a).
$$
By Lemma~\ref{long-add} it follows that
\begin{multline*}
X(u,\,0,\,a)Y(\tilde{\tilde u},\,0,\,b)=\\
=Y(\tilde{\tilde u},\,0,\,a+b)Y(v',\,0,\,a)Y(\tilde{\tilde u},\,v'a,\,0)Y(v,\,0,\,a)Y(\tilde{\tilde u},\,va,\,0)Y(v,\,v'a,\,0).
\end{multline*}
The same argument as above shows that $Y(v',\,0,\,b)$ commutes with both $X(u,\,0,\,a)$ and $Y(\tilde{\tilde u},\,0,\,a+b)$, so that again by Lemma~\ref{long-add}
\begin{multline*}
X(u,\,0,\,a)Y(\tilde{\tilde u},\,0,\,b)Y(v',\,0,\,b)=Y(\tilde{\tilde u},\,0,\,a+b)Y(v',\,0,\,a+b)\cdot\\
\cdot Y(\tilde{\tilde u},\,v'a,\,0)Y(v,\,0,\,a)Y(\tilde{\tilde u},\,va,\,0)Y(v,\,v'a,\,0).
\end{multline*}
Repeating this procedure, one eventually obtains the desired result.
\end{proof}

\begin{lm}
\label{long-scalar}
For $u\in V$ such that $u_i=u_{-i}=u_j=u_{-j}=0$ and $a,\,b\in R$, one has
$$
Y(ub,\,0,\,a)=Y(u,\,0,\,ab^2).
$$
\end{lm}

\begin{proof}
Decompose $Y(ub,\,0,\,a)$ as follows
\begin{multline*}
Y_{(i)}(ub,\,0,\,a)=[Y(e_i,\,ub,\,0),\,Y(e_{-i},\,0,\,a)]Y(e_{-i},\,uba\eps{-i},\,0)=\\
=[Y_{(j)}(-u,\,-e_ib),\,Y(e_{-i},\,0,\,a)]Y(e_{-i},\,uba\eps{-i},\,0)=\\
=[[Y(e_j,\,-u,\,0),\,Y(e_{-j},\,-e_ib\eps j,\,0)],\,Y(e_{-i},\,0,\,a)]Y(e_{-i},\,uba\eps{-i},\,0).
\end{multline*}
Now, we use the Hall--Witt identity 
$$\!\,^y[[y\inv,\,z],\,x]=\!\,^z[y,\,[z\inv,\,x]]\cdot\,^x[z,\,[x\inv,\,y]]$$
for $x=Y(e_{-i},\,0,\,a)$, $y=Y(e_j,\,u,\,0)$ and $z=Y(e_{-j},\,-e_ib\eps{j},\,0)$. Observe that
$$
Y_{(i)}(ub,\,0,\,a)=[[y\inv,\,z],\,x]Y(e_{-i},\,uba\eps{-i},\,0).
$$
Using the fact, that both $Y(e_i,\,ub,\,0)$, $Y(e_{-i},\,0,\,a)\in\ls j$ commute with $Y(e_j,\,u,\,0)$, we get
$$[[y\inv,\,z],\,x]=\!\,^y[[y\inv,\,z],\,x].$$
Next, $[x\inv,\,y]=1$ using conjugation formula for $x\inv$. Further,
\begin{multline*}
[z\inv,\,x]=z\inv\cdot\,^xz
=Y(e_{-j},\,e_ib\eps j,\,0)Y(e_{-j},\,e_ib\eps{-j}+e_{-i}a\eps{-i}b\eps{-j},\,0)=\\
=Y(e_{-j},\,e_{-i}ab\eps{i}\eps{j},\,ab^2)=Y(e_{-j},\,0,\,ab^2)Y(e_{-j},\,e_{-i}ab\eps{i}\eps{j},\,0).
\end{multline*}
Then, using $[a,\,bc]=[a,\,b]\cdot\!\,^b[a,\,c]$ one gets
\begin{multline*}
[y,\,[z\inv,\,x]]=[Y(e_j,\,u,\,0),\,Y(e_{-j},\,0,\,ab^2)Y(e_{-j},\,e_{-i}ab\eps{i}\eps{j},\,0)]=\\
=[Y(e_j,\,u,\,0),\,Y(e_{-j},\,0,\,ab^2)]\cdot\,^{Y(e_{-j},\,0,\,ab^2)}Y_{(j)}(u,\,e_{-i}ab\eps i,\,0)=\\
=[Y(e_j,\,u,\,0),\,Y(e_{-j},\,0,\,ab^2)]\cdot Y(e_{-i},\,uab\eps i,\,0).
\end{multline*}
Now,
\begin{multline*}
Y(ub,\,0,\,a)=\\
=\!\,^y[[y\inv,\,z],\,x]Y(e_{-i},\,uba\eps{-i},\,0)=\!\,^z[y,\,[z\inv,\,x]]Y(e_{-i},\,uba\eps{-i},\,0)=\\
=\!\,^z[Y(e_j,\,u,\,0),\,Y(e_{-j},\,0,\,ab^2)][Y(e_{-j},\,-e_ib\eps{j},\,0),\,Y(e_{-i},\,uab\eps i,\,0)]
\end{multline*}
At this point, using
\begin{multline*}
[Y(e_{-j},\,-e_ib\eps{j},\,0),\,Y(e_{-i},\,uab\eps i,\,0)]=Y_{(i)}(-e_{-j}b\eps j,\,uab,\,0)=\\
=Y(e_{-j},\,uab^2\eps{-j},\,0)=\!\,^zY(e_{-j},\,uab^2\eps{-j},\,0)
\end{multline*}
we finally obtain
\begin{multline*}
Y(ub,\,0,\,a)=\!\,^z[Y(e_j,\,u,\,0),\,Y(e_{-j},\,0,\,ab^2)]\cdot\!\,^zY_{(j)}(u,\,0,\,ab^2)=\\
=\!\,^zY_{(j)}(u,\,0,\,ab^2)=Y_{(j)}(u,\,0,\,ab^2).
\end{multline*}
\end{proof}

\begin{lm}
\label{x-long-scalar}
For $u\in V$, $a$, $b\in R$, one has
$$
X(ub,\,0,\,a)=X(u,\,0,\,ab^2).
$$
\end{lm}

\begin{proof}
Decomposing $X(ub,\,0,\,a)$ with the use of Lemmas~\ref{w=zz}, \ref{short-is-three-long} and \ref{short-symmetry}, and applying Lemmas~\ref{long-scalar} and \ref{short-symmetry}, we get
\begin{multline*}
X(ub,\,0,\,a)=Y(\tilde{\tilde u}b,\,0,\,a)Y(v'b,\,0,\,a)Y(\tilde{\tilde u}b,\,v'ba,\,0)Y(vb,\,0,\,a)\cdot\\
\cdot Y(\tilde{\tilde u}b,\,vba,\,0)Y(vb,\,v'ba,\,0)=Y(\tilde{\tilde u},\,0,\,ab^2)Y(v',\,0,\,ab^2)Y(\tilde{\tilde u},\,v'ab^2,\,0)\cdot\\
\cdot Y(v,\,0,\,ab^2)Y(\tilde{\tilde u},\,vab^2,\,0)Y(v,\,v'ab^2,\,0)=X(u,\,0,\,ab^2).
\end{multline*}

\end{proof}

Next, we have to define short-root ESD-generators.

\begin{df}
For $u$, $v\in V$ such that $\lan u,\,v\ran=0$, set
$$
X(v,\,u,\,0)=X(v,\,0,\,-1)X(u,\,0,\,-1)X(u+v,\,0,\,1).
$$
\end{df}

\begin{lm}
\label{short-obvious}
For $g\in\St\!\Sp(2l,\,R)$ and $u$, $v\in V$ such that $\lan u,\,v\ran=0$, $a\in R$, one has
\begin{enumerate}
\item $X(v,\,u,\,0)=X(u,\,v,\,0);$
\item $g\,X(v,\,u,\,0)g\inv=X(\phi(g)u,\,\phi(g)v,\,0);$
\item $X(u,\,ua,\,0)=X(u,\,0,\,2a)$.
\end{enumerate}
\end{lm}

\begin{proof}
Since {\it a}) and {\it b}) are obvious, it remains only to check {\it c}). By the very definition we have
$$
X(u,\,ua,\,0)=X(u,\,0,\,-1)X(ua,\,0,\,-1)X(ua+u,\,0,\,1).
$$
Then, using Lemma~\ref{x-long-scalar} and then Lemma~\ref{long-additivity}, we get
\begin{multline*}
X(u,\,0,\,-1)X(ua,\,0,\,-1)X(ua+u,\,0,\,1)=\\
=X(u,\,0,\,-1)X(u,\,0,\,-a^2)X(u,\,0,\,(a+1)^2)=X(u,\,0,\,2a).
\end{multline*}
\end{proof}

\begin{lm}
\label{new}
Consider $u$, $w\in V$ such that 
$$\lan u,\,w\ran=0,\qquad w_i=w_{-i}=w_j=w_{-j}=0.$$
Then
$$
X(u+w,\,0,\,a)=X(u,\,0,\,a)X(w,\,0,\,a)Y(w,\,ua,\,0).
$$
\end{lm}

\begin{proof}
Set $v=e_iu_i+e_{-i}u_{-i}$, $\tilde u=u-v$ and observe that $\lan w,\,v\ran=\lan w,\,\tilde u\ran=0$. By the very definition one has
$$
X(u+w,\,0,\,a)=Y_{(i)}(\tilde u+w,\,0,\,a)Y(v,\,0,\,a)Y(v,\,(\tilde u+w)a,\,0).
$$
Further, set $v'=e_ju_j+e_{-j}u_{-j}$ and $\tilde{\tilde u}=\tilde u-v'$. Since $X(\hat u,\,0,\,\hat a)$ is well-defined by Lemma~\ref{correctness}, one has
\begin{multline*}
Y_{(i)}(\tilde u+w,\,0,\,a)=X(\tilde u+w,\,0,\,a)=\\
=Y_{(j)}(\tilde{\tilde u}+w,\,0,\,a)Y(v',\,0,\,a)Y(v',\,(\tilde{\tilde u}+w)a,\,0).
\end{multline*}
Moreover,
$$
Y(\tilde{\tilde u}+w,\,0,\,a)=Y(\tilde{\tilde u},\,0,\,a)Y(w,\,0,\,a)Y(w,\,\tilde{\tilde u}a,\,0)
$$
by Lemma~\ref{short-is-three-long} and 
$$
Y(v',\,(\tilde{\tilde u}+w)a,\,0)=Y(v',\,\tilde{\tilde u}a,\,0)Y(v',\,wa,\,0)
$$
by Lemma~\ref{lm0}, whereas Lemma~\ref{w=zz} implies that
$$
Y(v,\,(\tilde u+w)a,\,0)=Y(v,\,(\tilde{\tilde u}+w)a,\,0)Y(v,\,v'a,\,0).
$$
Comparing these formulae, one gets
\begin{multline*}
X(u+w,\,0,\,a)=\\
=Y(\tilde{\tilde u},\,0,\,a)Y(w,\,0,\,a)Y(w,\,\tilde{\tilde u}a,\,0)Y(v',\,0,\,a)Y(v',\,\tilde{\tilde u}a,\,0)\cdot\\
\cdot Y(v',\,wa,\,0)Y(v,\,0,\,a)Y(v,\,\tilde{\tilde u}a,\,0)Y(v,\,wa,\,0)Y(v,\,v'a,\,0).
\end{multline*}
Invoking the usual arguments, it is easy show that the factors in the above product can be reordered as follows  
\begin{multline*}
X(u+w,\,0,\,a)=\\
=Y(\tilde{\tilde u},\,0,\,a)Y(v',\,0,\,a)Y(v',\,\tilde{\tilde u}a,\,0)Y(v,\,0,\,a)Y(v,\,\tilde{\tilde u}a,\,0)\cdot\\
\cdot Y(v,\,v'a,\,0)Y(w,\,0,\,a)Y(w,\,\tilde{\tilde u}a,\,0)Y(v,\,wa,\,0)Y(v',\,wa,\,0).
\end{multline*}
Finally, recall that 
\begin{multline*}
Y(\tilde{\tilde u},\,0,\,a)Y(v',\,0,\,a)Y(v',\,\tilde{\tilde u}a,\,0)Y(v,\,0,\,a)Y(v,\,\tilde{\tilde u}a,\,0)Y(v,\,v'a,\,0)=\\
=X(u,\,0,\,a)
\end{multline*}
and that
\begin{multline*}
Y(w,\,\tilde{\tilde u}a,\,0)Y(v,\,wa,\,0)Y(v',\,wa,\,0)=\\
=Y(w,\,\tilde{\tilde u}a,\,0)Y(w,\,va,\,0)Y(w,\,v'a,\,0)=Y(w,\,ua,\,0)
\end{multline*}
by Lemmas~\ref{short-symmetry} and \ref{w=zz}.
\end{proof}

The next lemma is an easy corollary of the previous one.

\begin{lm}
\label{x=y}
For $v\in V$ such that $v_{-i}=0$, one has
$$
X(e_i,\,v,\,0)=Y(e_i,\,v,\,0).
$$
\end{lm}

\begin{proof}
In the statement of Lemma~\ref{new} take $u=v$, $w=e_i$, $a=1$.
\end{proof}

\begin{lm}
\label{x-long-is-three-short}
Consider $a\in R$ and $u$, $v\in V$ such that $\lan u,\,v\ran=0$ and assume also that $v$ is a column of a symplectic elementary matrix. Then
$$
X(u+v,\,0,\,a)=X(u,\,0,\,a)X(v,\,0,\,a)X(v,\,ua,\,0).
$$
\end{lm}

\begin{proof}
Take $g\in\St\!\Sp(2l,\,R)$ such that $\phi(g)v=e_i$. Then,
$$
g\,X(u+v,\,0,\,a)g\inv=X(\phi(g)u+e_i,\,0,\,a).
$$
Now, Lemma~\ref{new} (and Lemma~\ref{x=y}) imply that
\begin{multline*}
X(\phi(g)u+e_i,\,0,\,a)=X(\phi(g)u,\,0,\,a)X(e_i,\,0,\,a)X(e_i,\,\phi(g)ua,\,0)=\\
=g\,X(u,\,0,\,a)X(v,\,0,\,a)X(v,\,ua,\,0)g\inv.
\end{multline*}
\end{proof}

The next lemma is an obvious consequence of the previous one.

\begin{lm}
Assume that both $u$ and $v$ are columns of symplectic elementary matrices such that $\lan u,\,v\ran=0$. Then for any $a\in R$ one has
$$
X(u,\,va,\,0)=X(v,\,ua,\,0).
$$
\end{lm}

\begin{lm}
\label{short-additivity}
Let $u\in V$ be column of a symplectic elementary matrix and let $v$, $w\in V$ be arbitrary columns such that $\lan u,\,v\ran=\lan u,\,w\ran=0$. Then
$$
X(u,\,v,\,0)X(u,\,w,\,0)=X(u,\,v+w,\,0)X(u,\,0,\,\lan v,\,w\ran).
$$
\end{lm}

\begin{proof}
Use the same trick as in Lemma~\ref{x-long-is-three-short}.
\end{proof}

\begin{df}
For $u$, $v\in V$ such that $\lan u,\,v\ran=0$, $a\in R$, set
$$
X(u,\,v,\,a)=X(u,\,v,\,0)X(u,\,0,\,a).
$$
\end{df}

\begin{lm}
\label{p-relations}
Assume that $u$, $u'$ are columns of symplectic elementary matrices, and let $v$, $v'$, $w\in V$ be arbitrary columns such that $\lan u,\,v\ran=\lan u,\,w\ran=0$, $\lan u',\,v'\ran=0$. Then for any $a$, $b\in R$ one has
\begin{align*}
&\text{a{\rm)} }X(u,\,v,\,a)X(u,\,w,\,b)=X(u,\,v+w,\,a+b+\lan v,\,w\ran),\\
&\begin{aligned}\text{b{\rm )} }&\text{If $v$ is also a column of a symplectic elementary matrix, then}\\&X(u,\,va,\,0)=X(v,\,ua,\,0),\end{aligned}\\
&\begin{aligned}\text{c{\rm)} }X(u',\,v',\,b)X(u,\,v,\,a)X(u',\,v',\,b)\inv&=\\&=X(T(u',\,v',\,b)u,\,T(u',\,v',\,b)v,\,a).\end{aligned}
\end{align*}
\end{lm}

\begin{proof}
For {\it a}) use Lemmas~\ref{short-additivity} and \ref{long-additivity}, {\it b}) and {\it c}) were already checked.
\end{proof}

\section{Symplectic van der Kallen group}

\begin{df}
Let the symplectic van der Kallen group $\St\!\Sp^*\!(2l,\,R)$ be the group defined by the set of generators 
\begin{multline*}
\big\{X^*(u,\,v,\,a)\,\big| u,\,v\in V,\ \text{$u$ is a column of}\\ \text{a symplectic elementary matrix},\ \lan u,\,v\ran=0,\ a\in R\big\}
\end{multline*}
and relations
\setcounter{equation}{0}
\renewcommand{\theequation}{P\arabic{equation}}
\begin{align}
&X^*(u,\,v,\,a)X^*(u,\,w,\,b)=X^*(u,\,v+w,\,a+b+\lan v,\,w\ran),\\
&\begin{aligned}X^*(u,\,va,\,0)=X^*(v,\,ua,\,0)\,\text{ whe}&\text{re $v$ is also a column}\\ &\text{of a symplectic elementary matrix,}\end{aligned}\\
&\begin{aligned}X^*(u',\,v',\,b)X^*(u,\,v,\,a)X^*(u',\,v&',\,b)\inv=\\&\!\!=X^*(T(u',\,v',\,b)u,\,T(u',\,v',\,b)v,\,a),\end{aligned}
\end{align}
\end{df}

\begin{rk}
Clearly, Lemma~\ref{p-relations} amounts to the existence of a homomorphism 
$$\varpi\colon\St\!\Sp^*\!(2l,\,R)\epi\St\!\Sp(2l,\,R),$$
sending $X^*(u,\,v,\,a)$ to $X(u,\,v,\,a)$, which is obviously surjective. Furthermore, P3 implies that $\varpi$ is in fact a central extension.
\end{rk}

We have to construct the inverse isomorphism from the Steinberg group to the van der Kallen group. Let us start with the following lemma.

\begin{lm}
\label{r3-r5}
For $u$, $v\in V$, where $u$ is column of a symplectic elementary matrix, such that $\lan u,\,v\ran=0$, $u_i=u_{-i}=v_i=v_{-i}=0$, and for any $a$, $b\in R$ one has
$$
[X^*(e_i,\,ub,\,0),\,X^*(e_{-i},\,v\eps i,\,a)]=X^*(u,\,vb,\,ab^2)X^*(e_{-i},\,uab\eps i,\,0).
$$
\end{lm}

\begin{proof}
Observe that $X^*(\hat u,\,\hat v,\,0)\inv=X^*(\hat u,\,-\hat v,\,0)$ by P1, so that 
\begin{multline*}
[X^*(e_i,\,ub,\,0),\,X^*(e_{-i},\,v\eps i,\,a)]=\\=X^*(u,\,e_ib,\,0)\cdot\!\,^{X^*(e_{-i},\,v\eps i,\,a)}X^*(u,\,-e_ib,\,0)
\end{multline*}
by P2. Then, 
\begin{multline*}
X^*(u,\,e_ib,\,0)\cdot\!\,^{X^*(e_{-i},\,v\eps i,\,a)}X^*(u,\,-e_ib,\,0)=\\=X^*(u,\,e_ib,\,0)X^*(u,\,-e_ib+e_{-i}ab\eps{i}+vb,\,0)
\end{multline*}
by P3 and finally
\begin{multline*}
X^*(u,\,e_ib,\,0)X^*(u,\,-e_ib+e_{-i}ab\eps{i}+vb,\,0)=\\=X^*(u,\,e_{-i}ab\eps i+vb,\,ab^2)=X^*(u,\,vb,\,ab^2)X^*(u,\,e_{-i}ab\eps i,\,0)
\end{multline*}
by P1. Now, the claim is obvious by P2.
\end{proof}

\begin{df}
Set
\begin{align*}
&X^*_{ij}(a)=X^*(e_i,\,e_{-j}a\eps{-j},\,0) \text{ for } i\not\in\{\pm j\},\\ 
&X^*_{i,\,-i}(a)=X^*(e_i,\,0,\,a).
\end{align*}
\end{df}

\begin{lm}
\label{st-relations}
Steinberg relations {\rm R0}--{\rm R5} hold for $X^*_{ij}(a)$ and $X^*_{i,-i}(a)$.
\end{lm}

\begin{proof}
Indeed, P2 implies R0 and P1 implies R1. To establish R1 one has to consider two cases, where $j\neq-i$ and where $j=-i$, respectively. To establish R2 one proceeds as follows. Firstly, one has to consider three cases, where $j\neq-i$, $k\neq-h$, where $j\neq-i$, $k=-h$, and finally, where $j=-i$, $k=-h$. In each case using P3 show that $X^*_{ij}(a)$ commutes with $X^*_{hk}(b)$. Relations R3 and R4 follow directly from Lemma~\ref{r3-r5}. More precisely, it is easier to obtain a version of R4, where the factors in the commutator are interchanged, rather than R4 itself. Finally, we proceed with R5. Using P2 and then P1 we get
\begin{multline*}
[X^*_{ij}(a),\,X^*_{j,-i}(b)]=[X^*(e_i,\,-e_{-j}a\eps j,\,0),\,X^*(e_i,\,-e_jb\eps{-i},\,0)]=\\=X^*(e_i,\,0,\,2ab\eps i).
\end{multline*}
\end{proof}

\begin{rk}
By Lemma~\ref{r3-r5} it follows that $$X^*(e_i,\,0,\,2a)=X^*_{i,-i}(2a)=[X^*_{ij}(a),\,X^*_{j,-i}(\eps i)]=X^*(e_i,\,e_ia,\,0).$$
\end{rk}

\begin{cl*}
There is a homomorphism $$\varrho\colon\St\!\Sp(2l,\,R)\rightarrow\St\!\Sp^*\!(2l,\,R)$$ sending $X_{ij}(a)$ to $X^*_{ij}(a)$. Obviously, $\varpi\varrho=1$, i.e. $\varrho$ is a splitting for $\varpi$.
\end{cl*}

It remains only to verify that $\varrho\varpi=1$. From a technical viewpoint, it is much more convenient not to proceed directly, but to invoke the following fact (see~\cite{vdK,L2} for the proof).

\begin{lm}
Let $\pi\colon G\epi H$ be a splitting central extension with $G$ perfect. Then $\pi$ is in fact an isomorphism.
\end{lm}

It remains to show that van der Kallen group is perfect.

\begin{lm}
\label{vdk-unipotent-decomposition}
For $v\in V$ such that $v_{-i}=0$, set 
$$v_-=\sum_{i<0}e_iv_i\qquad\text{and similarly}\qquad v_+=\sum_{i>0}e_iv_i.$$ 
Then
\begin{multline*}
X^*(e_i,\,v,\,a)=X^*_{i,-i}(a+2v_i-v_i\eps i-\lan v_-,\,v_+\ran)\cdot\\ \cdot X^*_{-l,-i}(v_{-l}\eps i)\ldots X^*_{-1,-i}(v_{-1}\eps i)X^*_{1,-i}(v_1\eps i)\ldots X^*_{l,-i}(v_l\eps i).
\end{multline*}
\end{lm}

\begin{proof}
The proof of Lemma~\ref{unipotent-decomposition} works in this situation almost verbatim. One only has to recall that $[X^*(e_i,\,0,\,b),\,X^*(e_i,\,\hat v,\,\hat a)]=1$ by P1 and that $X^*(e_i,\,0,\,2a)=X^*(e_i,\,e_ia,\,0)$.
\end{proof}

\begin{lm}
The group $\St\!\Sp^*\!(2l,\,R)$ is perfect.
\end{lm}

\begin{proof}
Firstly, observe that $X^*_{ij}(a)$ lie in $[\St\!\Sp^*\!(2l,\,R),\,\St\!\Sp^*\!(2l,\,R)]$ by relations R3 and R4. Thus, $X^*(e_i,\,v,\,a)$ also lie in the commutator subgroup by Lemma~\ref{vdk-unipotent-decomposition}. Now, P3 implies that any generator $X^*(u,\,v,\,a)$ lies in the commutator subgroup. Indeed, this is obvious since $u$ is a column of a symplectic elementary matrix.
\end{proof}

Now, it follows that $\St\!\Sp^*\!(2l,\,R)\cong\St\!\Sp(2l,\,R)$.

\bibliography{stsp}{}
\bibliographystyle{plain}

\end{document}